\documentclass[10pt,a4paper]{amsart}
\usepackage{amssymb}
\usepackage{verbatim}
\usepackage[T1]{fontenc}
\usepackage{graphicx}
\usepackage{enumerate}
\usepackage{amsmath,amsfonts,amssymb}
\usepackage{color}
\usepackage{cite}
\definecolor{citation}{rgb}{0.11,0.67,0.84}
\definecolor{formula}{rgb}{0.1,0.2,0.6}
\definecolor{url}{rgb}{0.11,0.67,0.84}
\usepackage{pgf,tikz}
\usepackage{mathrsfs}
\usepackage{fancyhdr}
\usepackage{dutchcal}

\usepackage{dsfont}

\newcommand{\reqnomode}{\tagsleft@false}
\usepackage[colorlinks,pdfpagelabels,pdfstartview = FitH,bookmarksopen = true,bookmarksnumbered = true,urlcolor=url,linkcolor = formula,plainpages = false,hypertexnames = false,citecolor = citation] {hyperref}

\def\dx{\,{\rm d}x}

\def\dy{\,{\rm d}y}
\def\dz{\,{\rm d}z}
\def \d{\,{\rm d}}

\def\dist{\,{\rm dist}}

\def\supp{\,{\rm supp}}

\def\er{\mathbb R}
\allowdisplaybreaks
\makeatletter
\DeclareRobustCommand*{\bfseries}{%
  \not@math@alphabet\bfseries\mathbf
  \fontseries\bfdefault\selectfont
  \boldmath
}

\makeatother

\newlength{\defbaselineskip}
\newcommand{\setlinespacing}[1]
           {\setlength{\baselineskip}{#1 \defbaselineskip}}
\newcommand{\mint}{\mathop{\int\hskip -1,05em -\, \!\!\!}\nolimits}

\newtheorem{theorem}{Theorem}
\newtheorem{corollary}{Corollary}
\newtheorem{definition}{Definition}
\newtheorem{remark}{Remark}
\newtheorem{lemma}{Lemma}
\newtheorem{proposition}{Proposition}
\newcommand{\N}{\mathbb{N}}
\newcommand\eps\varepsilon
\def\en{\mathbb N}

\def\er{\mathbb R}

\def\mean#1{\mathchoice%
          {\mathop{\kern 0.2em\vrule width 0.6em height 0.69678ex depth -0.58065ex
                  \kern -0.8em \intop}\nolimits_{\kern -0.4em#1}}%
          {\mathop{\kern 0.1em\vrule width 0.5em height 0.69678ex depth -0.60387ex
                  \kern -0.6em \intop}\nolimits_{#1}}%
          {\mathop{\kern 0.1em\vrule width 0.5em height 0.69678ex
              depth -0.60387ex
                  \kern -0.6em \intop}\nolimits_{#1}}%
          {\mathop{\kern 0.1em\vrule width 0.5em height 0.69678ex depth -0.60387ex
                  \kern -0.6em \intop}\nolimits_{#1}}}

\numberwithin{equation}{section}
\allowdisplaybreaks[4]

\newcommand{\rr}{\varrho}
\newcommand{\snr}[1]{\lvert #1\rvert}
\newcommand{\nr}[1]{\lVert #1 \rVert}
\newcommand{\rif}[1]{(\ref{#1})}
\newcommand{\stackleq}[1]{\stackrel{\rif{#1}}{ \leq}}

\def\loc{\operatorname{loc}}

\def\eqn#1$$#2$${\begin{equation}\label#1#2\end{equation}}

\newcommand{\data}{\textit{\texttt{data}}}

\newcommand\ttau{\tilde \tau}

\makeatother

\title[Interpolative gap bounds for nonautonomous integrals]{Interpolative gap bounds for nonautonomous integrals}
\author[De Filippis]{Cristiana De Filippis}  \address{Cristiana De Filippis\\Dipartimento di Matematica "Giuseppe Peano", Universit\`a di Torino\\ Via Carlo Alberto 10, 10123 Torino, Italy} \email{\url{cristiana.defilippis@unito.it}}
\author[Mingione]{Giuseppe Mingione}  \address{Giuseppe Mingione\\Dipartimento SMFI, Universit\`a di Parma, Viale delle Scienze 53/a, Campus, 43124 Parma, Italy} \email{\url{giuseppe.mingione@unipr.it}}
\begin{document}
\subjclass[2010]{35J60, 35J70\vspace{1mm}} 

\keywords{Regularity, non-autonomous functionals, $(p,q)$-growth\vspace{1mm}}

\thanks{{\it Acknowledgements.}\ This work is supported by the University of Turin via the project "Regolarit\'a e propriet\'a qualitative delle soluzioni di equazioni alle derivate parziali" and by the University of Parma via the project ``Regularity, Nonlinear Potential Theory and related topics".
\vspace{1mm}}

\maketitle

\centerline{To Vladimir Gilelevich Maz'ya, master of Analysis}

\begin{abstract}
For nonautonomous, nonuniformly elliptic integrals with so-called $(p,q)$-growth conditions, we show a general interpolation property allowing to get basic higher integrability results for H\"older continuous minimizers under improved bounds for the gap $q/p$. For this we introduce a new method, based on approximating the original, local functional, with mixed local/nonlocal ones, and allowing for suitable estimates in fractional Sobolev spaces. 
\end{abstract}
\vspace{3mm}

\setlinespacing{1.08}
\maketitle

\section{Introduction}
In this paper we give new contributions to the regularity theory of minima of integral functionals of the type 
\eqn{genF}
$$
W^{1,1}_{\loc}(\Omega,\mathbb{R}^{N}) \ni w   \mapsto\mathcal F(w, \Omega):= \int_{\Omega} F(x, Dw) \dx\,.
$$
Here, as in the following, $\Omega \subset \er^n$ will denote an open subset, where $n\geq 2$ and $N\geq 1$; the function $F \colon \Omega \times \er^{N\times n} \to [0, \infty)$ will always be Carath\'edory integrand. The definition of (local) minimizer we use here is standard in the literature and it is given by 
\begin{definition}\label{defi-min} A map $u \in W^{1,1}_{\loc}(\Omega,\er^N)$ is a local minimizer of the functional $\mathcal F$ in~\eqref{genF} if, for every open subset $\tilde \Omega\Subset \Omega$, we have $\mathcal F(u;\tilde \Omega) <\infty$ and $\mathcal F(u;\tilde \Omega)\leq \mathcal F(w;\tilde \Omega)$ holds for every competitor $w \in u + W^{1,1}_0(\tilde \Omega; \er^N)$. 
\end{definition}
We shall abbreviate local minimizer simply by minimizer. The main point in this paper is that the integrand $F:\Omega \times \er^{N\times n} \to [0, \infty)$ is both nonautonomous and nonuniformly elliptic. Specifically, following the notation used in \cite{BM, ciccio,DM}, we are in the situation when the {\em ellipticity ratio}
$$\mathfrak {R}_{F}(z,B)\equiv \mathfrak {R}(z,B):=\frac{\sup_{x\in B}\, \mbox{highest eigenvalue of} \ \partial_{zz}F(x,z)}{\inf_{x\in B}\, \mbox{lowest eigenvalue of} \ \partial_{zz}F(x,z)}$$
($B \subset \Omega$ is any ball), is such that $\mathfrak {R}_{F}(z,B)\to \infty$ when $|z|\to \infty$ for at least one ball $B$. This happens for instance in the paramount example given by the double phase functional
\eqn{pqfunctional}
$$
w \mapsto \int_{\Omega} \left[|Dw|^p+a(x)|Dw|^{q}\right]\dx\,,\ \  \mbox{where} \ \ 1 < p < q\,, \ 0 \leq a(\cdot)\in C^{\alpha}(\Omega)\,, \ \ \alpha \in (0,1]\;.$$
Indeed, on a ball $B$ such that $B\cap \{a(x)=0\}\not =0$, in the case of \rif{pqfunctional} we have
\eqn{dopo}
$$\mathfrak {R}(z,B)\approx \|a\|_{L^{\infty}(B)}|z|^{q-p}+1\;.$$
The functional in \rif{pqfunctional} has been introduced by Zhikov \cite{Z1, Z2, Z3, Z4} in the setting of Homogenization theory \cite{Z4}. For minima of \rif{pqfunctional}, there is a by a now a rather complete regularity theory, initiated many years ago in \cite{ELM}; see \cite{BCM} for the most updated statements. In particular, it has been proved that the condition 
\eqn{maincond}
$$
\frac{q}{p} \leq 1+\frac{\alpha}{n}
$$
is necessary \cite{ELM, FMM} and sufficient \cite{BCM} for the local H\"older gradient continuity of minima. In particular, failure of \rif{maincond} implies that in general, minimizers, that by definition are $W^{1,p}$-regular, do not belong to $W^{1,q}$. Therefore even the initial integrability bootstrap fails. The bound in \rif{maincond} reflects the delicate balance between the smallness of $a(\cdot)$ around $\{a(x)=0\}\cap B$, and the growth of $F(x,z)$ with respect to the gradient variable $z$, which is necessary to keep $\mathfrak {R}(z,B)$ under control when proving a priori estimates for minima. We refer to \cite{BCM, CM2} for a larger discussion. 

Functionals of the type in \rif{pqfunctional} fall in the realm of so-called functionals with $(p,q)$-growth, i.e. those satisfying unbalanced ellipticity conditions of the type (for $|z|$ large)
\eqn{seconde}
$$
|z|^{p-2} \mathds{I}_{\rm d} \lesssim \partial_{zz} F(\cdot, z) \lesssim  |z|^{q-2} \mathds{I}_{\rm d} \Longrightarrow  \mathfrak {R}_{F}(z,B) \lesssim |z|^{q-p}
$$
for $1< p \leq  q$. These have been considered in a special case by Uraltseva \& Urdaletova  \cite{UU}, and then, systematically, in the seminal papers by Marcellini \cite{M0, M1, M2}. The case $p=q$ falls therefore in the realm of uniformly elliptic problems. A large literature devoted to the study of regularity theory for such functionals is nowadays available; we refer to the surveys \cite{M3, Majmaa, dark} for a reasonable overview. In this context, the ratio $q/p$ is usually called gap. A distinctive feature of such functionals is that, in general, gap bounds of the type 
\eqn{generalbound}
$$
\frac{q}{p} < 1 + {\rm o}(n)\,, \qquad    {\rm o}(n)\approx \frac 1n
$$
are necessary and sufficient for regularity of minima \cite{ELM, M1}. More recent developments concerning bounds of the type in \rif{generalbound} are in \cite{BS, BB, BF, BO, dl, doh, HS, HH, HHO, HO, S}. Notice that the bound appearing in \rif{maincond} is of the type in \rif{generalbound}. In this situation 

An interesting phenomenon of interpolative nature appears when considering a priori more regular minimizers. For instance, assuming that minima are bounded, leads to non-dimensional bounds on the distance $q-p$, that can be made independent of $n$; see for instance \cite{BCM, CKP, choe}. In particular, in \cite{CM2} the authors have proved that assuming that minima are bounded allows to replace \rif{maincond} by 
\eqn{boulim}
$$
q \leq p+\alpha\,.
$$
This is better than \rif{maincond} provided $p\leq n$ (that is when boundedness of minima is not automatically implied by Sobolev-Morrey embedding, and it is therefore a genuine assumption). The bound in \rif{boulim} is optimal as shown by the counterexamples in \cite{ELM, FMM}. A partial generalization of the result in \cite{BCM} has been obtained in \cite{DM}, under the same bound in \rif{boulim}, with strict inequality. More recently, in the specific case of the double phase functional \rif{pqfunctional}, in \cite{BCM} the authors have proved that, assuming a priori $C^{0,\gamma}$-regularity, minimizers are regular provided
\eqn{asy0}
$$
q < p+\frac{\alpha}{1-\gamma}\,,
$$
thus providing a further weakened gap bound. This condition is sharp too, as recently proved in \cite{BDS}. It gives back \rif{boulim} for $\gamma \to 0$. 
In particular, notice that the asymptotic of \rif{asy0} for $\gamma \to 1$ is of the type
\eqn{asy}
$$
q
< p + {\rm O}(\gamma)\,, \qquad    {\rm O}(\gamma)\approx  \frac{1}{1-\gamma}\;.$$
This is in accordance with the fact that the focal point in regularity for $(p,q)$-problems is Lipschitz continuity of minima. Once this is achieved, the functional in question goes back to the realm of uniformly elliptic ones as growth conditions for large $|z|$ become irrelevant; see also \cite[Section 6]{DM}. Bounds of the type in \rif{asy0} have an interpolative nature. In fact, in the scheme of Caccioppoli type inequalities coming up in regularity estimates, the a priori, assumed regularity on minima, allows for a better control of nonuniform ellipticity. 

The above mentioned result in \cite{BCM} holds in the scalar case $N=1$ and for the specific functional in \rif{pqfunctional}; these facts  play a crucial role in the analysis there. Therefore the question arises of whether or not bounds with asymptotics as in \rif{asy} imply regularity of $C^{0, \gamma}$-minima for general functionals with $(p,q)$-growth. The question is open already in the autonomous case $F(x, Dw)\equiv F(Dw)$. We are here able to give a first, positive answer by showing that bounds as in \rif{asy} are in general sufficient to prove the basic step of improving the regularity of $C^{0, \gamma}$-minima from $W^{1,p}$ to $W^{1,q}$. See condition \rif{tip} below. This initial integrability bootstrap is usually the starting point from proving higher regularity - see \cite{ELM, DM} - and it is essentially the best possible result in the vectorial case considered here.

\subsection{Statements of the results} In order to quantify ellipticity, we shall use assumptions that are more general of those using the Hessian of $F(\cdot)$, as in \rif{seconde}, following the approach in \cite{DM, ELM}. Indeed, we shall consider an integrand $F(\cdot)$ which is H\"older continuous in the space variable $x$ and differentiable in the gradient variable $z$, and such that $\partial_z F(\cdot)$ is still a Car\'atheodory integrand. Monotonicity/ellipticity and growth assumptions are described by requiring that
\eqn{assF}
$$
\left\{
\begin{array}{c}
\nu |z|^{p}\leq F(x,z) \leq L(|z|^{q}+1)\\ [6 pt]
\nu\left(|z_{1}|^{2} +
  |z_{2}|^{2}+\mu^{2}\right)^{(p-2)/2} |z_{1}-z_{2}|^{2} \leq
\left( \partial_z F(x,z_{1})-\partial_z F(x,z_{2})\right)\cdot\left(z_{1}-z_{2}\right) \\ [6 pt]
|\partial_z F(x,z)-\partial_z F(y,z)|
\leq L |x-y|^{\alpha}(|z|^{q-1}+1)\;,
 \end{array}\right.
 $$
hold whenever $x,y\in \Omega$, $z, z_1, z_2 \in \er^{N\times n}$, where $1<p\leq q $, $\mu \in [0,1]$, $\alpha \in (0,1]$ and $0< \nu \leq 1 \leq L$ are fixed constants. Notice that \rif{assF}$_1$ implies that minimizers of $\mathcal F$ are automatically locally $W^{1,p}$-regular. As we are dealing with a nonautonomous functional, the so-called Lavrentiev phenomenon naturally comes into the play \cite{Mlav, Z1, Z2, Z3}. Its nonoccurrence is a necessary condition for regularity of minima. Therefore we are led to consider a functional, called Lavrentiev gap, providing a quantitative measure of such a phenomenon. We refer the reader to \cite{ABF, DM, ELM} for more information. Since we are in fact interested in the regularity of \emph{a priori} H\"older continuous minima of the functional \eqref{genF}, we shall adopt suitable definitions of relaxed and gap functionals aiming at exploiting this fact. For this, let us consider a functional of the type in \rif{genF}, where the integrand $F(\cdot)$ satisfies \rif{assF}, and a number $H> 0$. Given a ball $B \subset \Omega$, the natural relaxation of $\mathcal{F}$ we consider here is
\eqn{lav0}
$$
\bar{\mathcal{F}}_H(w,B):=\inf_{\{w_{j}\}\in \mathcal{C}_H(w,B)}\left\{\liminf_{j\to \infty}\int_{B}F(x,Dw_{j}) \, dx\right\},
$$
defined for any $w \in W^{1,1}(B, \er^n)$, where (see \rif{notationh} below for the notation)
\eqn{ilC}
$$
\mathcal{C}_H(w,B):=\left\{\{w_{j}\}\subset W^{1,\infty}(B,\mathbb{R}^{N})\colon w_{j}\rightharpoonup w \ \mbox{in} \ W^{1,p}(B,\mathbb{R}^{N}), \ \sup_{j}\, [w_{j}]_{0,\gamma;B} \leq H \right\}.
$$  
Accordingly, as in \cite{ABF, DM, ELM}, we consider the Lavrentiev gap functional
\eqn{lav1}
$$
\mathcal {L}_{F,H}(w,B):= \bar{\mathcal{F}}_H(w,B)- \mathcal{F}(w,B)\,,
$$
defined for every $w \in W^{1,1}(B)$ such that $ \mathcal{F}(w,B)$ is finite; we set $\mathcal {L}_{F,H}(w,B)=0$ otherwise. 
\begin{remark}\label{bullets}{\em 
We here collect a few immediate consequences of the above definitions. 
\begin{itemize}
\item The convexity of $F(\cdot)$, implied by \rif{assF}$_2$, provides that the functional $\mathcal F$ in \rif{genF} is lower semicontinuous with respect to the weak convergence of $W^{1,p}$. It follows that 
$\mathcal {L}_{F,H}(\cdot,B)\geq 0$ holds whenever $w \in W^{1,1}(B, \er^N)$ and $B \Subset\Omega$ is a ball.  
\item Consider a map $w\in W^{1,p}_{\loc}(\Omega,\mathbb{R}^{N})\cap C^{0,\gamma}_{\loc}(\Omega,\mathbb{R}^{N})$. A simple mollification argument then shows that $\mathcal{C}_H(w,B)$ is non empty whenever $B \Subset \Omega$ for $H \approx \|w\|_{C^{0,\gamma}(B)}$. For this see also the proofs in Section \ref{convsec0}  below. This anticipates that, when considering a $C^{0,\gamma}$-regular minimizer $u$ of the functional $\mathcal F$, it will happen that we shall use $\mathcal {L}_{F,H}(u,B)$ with $H \approx \|u\|_{C^{0,\gamma}(B)}$ (see for instance Theorem \ref{t2} below, and compare \rif{stima2} with \rif{ine}). 
\item A straightforward consequence of \rif{lav0}-\rif{lav1} is
\begin{proposition} \label{approssi1}
 Let $w\in W^{1,1}(B, \er^N)$ be such that
$\mathcal{F}(w,B)$ is finite, where $B \subset \Omega$ is a fixed ball. Then 
$
\mathcal{F}(w,B)=\bar{\mathcal{F}}_H(w,B)$ for some $H>0$, if and only if there exists $\{w_{j}\}\in \mathcal{C}_H(w,B)$ such that 
\eqn{approenergia}
$$\mathcal{F}(w_{j},B)\to \mathcal{F}(w,B)\,.$$ 
\end{proposition}
\item Accordingly, if $\mathcal{C}_H(w,B)$ is empty, then $\bar{\mathcal{F}}_H(w,B)=\infty$; moreover, again in this case, if $ \mathcal{F}(w,B)$ is finite, then $\bar{\mathcal{F}}_H(w,B)=\mathcal {L}_{F,H}(w,B) =\infty$. This is in accordance with the fact that when $w$ is not $\gamma-$H\"older continuous, it is in general impossible to build a sequence from $\mathcal{C}_H(w,B)$, for any $H>0$, approximating $w$ in energy in the sense of  \rif{approenergia} below. Such approximations are in fact usually built via smooth convolutions when $w$ is already H\"older continuous. 
\item Conversely, if $\mathcal{C}_H(w,B)$ is non empty, then $w \in C^{0, \gamma}(B)$ and $[w]_{0, \gamma;B}\leq H$.
\item If $w$ is a locally $W^{1,q}\cap C^{0,\gamma}$-regular map, then a density and convolution argument gives that  $\mathcal {L}_{F,H}(w,B)=0$ holds for every ball $B \Subset \Omega$, with $H \approx \|w\|_{C^{0,\gamma}(B)}+1$. This last condition is therefore in a sense necessary to prove the local $W^{1,q}$-regularity of $C^{0,\gamma}$-regular minima of the original functional $\mathcal F$.
\end{itemize}}
\end{remark}
Accordingly to the last point in Remark \ref{bullets}, the main result of this paper is now
\begin{theorem}\label{t1}
Let $u\in W^{1,p}_{\loc}(\Omega,\mathbb{R}^{N})$ be a minimizer of functional \eqref{genF}, under assumptions \eqref{assF} and 
\begin{flalign}\label{pq}
q<p+\frac{\min\{\alpha,2\gamma\}}{\vartheta(1-\gamma)}\,,\qquad \mbox{where} \ \ \vartheta:=\begin{cases}
\ 1\quad &\mbox{if}\ \ p\ge 2\\
\ \frac{2}{p}\quad &\mbox{if} \ \ 1<p<2\,,
\end{cases}\, \qquad 0<\gamma <1\,.
\end{flalign}
Assume also that
\begin{flalign}\label{nolav}
\mathcal {L}_{F,H}(u,B_{r})=0
\end{flalign}
holds for a ball $B_{r}\Subset \Omega$ with $r\le 1$, and for some $H >0$. If $\mathfrak{q}$ is a number such that
\begin{flalign}\label{tip}
q\leq \mathfrak{q}<p+\frac{\min\{\alpha,2\gamma\}}{\vartheta(1-\gamma)}
\end{flalign}
and $B_{\rr}\Subset B_{r}$ is a ball concentric to $B_{r}$, then
\eqn{ine}
$$
\nr{Du}_{L^{\mathfrak{q}}(B_{\rr})}\le \frac{c}{(r-\rr)^{\kappa_{1}}}\left([\mathcal{F}(u,B_{r})]^{1/p}+H+1\right)^{\kappa_{2}}
$$
holds for constants $c\equiv c(n,p,q, \nu, L, \alpha,\gamma,\mathfrak q)$ and $  \kappa_{1},\kappa_{2}\equiv    \kappa_{1},\kappa_{2}(n,p,q,\alpha,\gamma,\mathfrak{q})$. In particular, if \eqref{nolav} holds for every ball $B_r\Subset \Omega$, then $u\in W^{1,q}_{\loc}(\Omega,\mathbb{R}^{N})$. 
\end{theorem}

\begin{remark} \emph{The reader might of course wonder where the a priori $C^{0, \gamma}$-regularity of the minimizer $u$ is assumed in Theorem \ref{t1}. This is hidden in assumption \rif{nolav}. In fact, as $F(\cdot, Du)\in L^1_{\loc}(\Omega)$ by minimality, it follows from the fourth and the fifth point of Remark \ref{bullets} that $\mathcal{C}_H(u,B_r)$ is non-empty and therefore $u \in C^{0, \gamma}(B_r,\er^N)$ (with $[u]_{0, \gamma;B_r}\leq H$). This said, the bound in \rif{tip} is exactly of the type in \rif{asy}. Let us now consider the case $p\geq2$. When $\alpha \leq 2\gamma$ the bound in \rif{pq} coincides with \rif{asy0}, that is the one considered in \cite{BCM} for the specific functional \rif{pqfunctional}.}
\end{remark}
As mentioned above, assumption \rif{nolav} is in a sense necessary to prove local $W^{1,q}$-regularity of minimizers of $\mathcal F$. As a matter of fact, condition \rif{nolav} is always satisfied in a large number of situations. A very relevant one is when the integrand is autonomous, i.e., $F(x, Du)\equiv F(Du)$. In such a case, the Lavrentiev gap disappears due to basic convexity arguments, and we have:
\begin{theorem}\label{t2}
Let $u\in W^{1,p}_{\loc}(\Omega,\mathbb{R}^{N})\cap C^{0,\gamma}_{\loc}(\Omega,\mathbb{R}^{N})$ be a minimizer of functional \eqref{genF}, where $0<\gamma < 1$, under assumptions \eqref{assF} with $F(x, z)\equiv F(z)$ and 
\begin{flalign}\label{pqa}
q<p+\frac{\min\{1,2\gamma\}}{\vartheta(1-\gamma)}\,,
\end{flalign}
where $\vartheta$ is as in \eqref{pq}. 
If $\mathfrak{q}$ is a number such that
\begin{flalign}\label{tipau}
q\leq \mathfrak{q}<p+\frac{\min\{1,2\gamma\}}{\vartheta(1-\gamma)}
\end{flalign}
and $B_{\rr}\Subset B_{r}\Subset \Omega$ are concentric balls, then 
\eqn{stima2}
$$
\nr{Du}_{L^{\mathfrak{q}}(B_{\rr})}\le \frac{c}{(r-\rr)^{\kappa_{1}}}\left([\mathcal{F}(u,B_{r})]^{1/p}+[u]_{0, \gamma;B_r}+1\right)^{\kappa_{2}}
$$
holds with $c, \kappa_1, \kappa_2$ as in \eqref{ine}. In particular, $u\in W^{1,q}_{\loc}(\Omega,\mathbb{R}^{N})$. 
\end{theorem}
Another situation when \rif{nolav} can be automatically satisfied, is when the integrand $F(\cdot)$ is equivalent to a convex function $G\colon \er^{N\times n} \to [0, \infty)$ modulo a multiplicative factor, i.e., 
\eqn{orli1}
$$
b(x)G(z) \lesssim F(x, z) \lesssim b(x)G(z) +1\,, \qquad 0 \leq b(\cdot),1/b(\cdot) \in L^{\infty}(\Omega)\,.
$$
\begin{corollary}\label{c0} Let $u\in W^{1,p}_{\loc}(\Omega,\mathbb{R}^{N})\cap C^{0,\gamma}_{\loc}(\Omega,\mathbb{R}^{N})$ be a minimizer of functional \eqref{genF} under the assumptions \eqref{assF} and \eqref{pq}. Furthermore, assume that \eqref{orli1} is satisfied too. Then 
\eqref{stima2}
holds whenever $B_{\varrho} \Subset B_r \Subset \Omega$ are concentric balls and for the range of exponents in \eqref{tip}. 
\end{corollary}
Back to the full nonautonomous case, a relevant example in this setting is given by \cite[Theorem 4]{BCM}. This deals with functionals modelled on the double phase functional in \rif{pqfunctional}, i.e., growth conditions as 
\eqn{contrdp}
$$
 |z|^{p}+a(x) |z|^{q}\lesssim F(x,z) \lesssim  |z|^{p}+a(x) |z|^{q} +1 \,, \qquad 0 \leq a(\cdot) \in C^{0, \alpha}(\Omega)\,,
$$
are assumed for every $(x, z)\in \Omega \times \er^{N\times n}$, no matter \rif{assF} are satisfied or not. Then \rif{asy0} guarantees that the approximation in energy \rif{approenergia} holds for a sequence of $W^{1,\infty}$-regular maps $\{w_j\}$, provided 
$w \in C^{0, \gamma}$ holds and the bound in \rif{asy0} is in force. This fact allows to draw another consequence from Theorem \ref{t1}, that is
\begin{corollary}\label{c1} Let $u\in W^{1,p}_{\loc}(\Omega,\mathbb{R}^{N})\cap C^{0,\gamma}_{\loc}(\Omega,\mathbb{R}^{N})$ be a minimizer of functional \eqref{genF} under the assumptions \eqref{assF} and \eqref{pq}. Furthermore, assume that \eqref{contrdp} is satisfied too. Then \eqref{stima2} holds for the range of exponents displayed in \eqref{tip}. 
\end{corollary}
More general cases of double sided bounds as in \rif{contrdp} for which similar corollaries hold can be found in \cite[Section 5]{ELM}. We refer to this last paper also a for larger discussion on the use of Lavrentiev gap functionals in this setting. 
\subsection{Novelties and techniques} 
The proof of Theorem \ref{t1} relies of three main ingredients. The first one is the use of a method aimed at approximating the original minimizer $u$ of $\mathcal F$, with higher integrable solutions to a different kind of variational problems. This is necessary as the starting lack of $W^{1,q}$-integrability of $u$ does not allow to use the Euler-Lagrange system of $\mathcal F$. We notice that the possibility of this approximation relies on assumption \rif{nolav}, that can be used, in a sense, to find good boundary values to build the approximating sequence of minimizers. Here we encounter a first difficulty as, in order to get useful a priori estimates, we have to preserve the property of being $C^{0, \gamma}$-regular, transferring at least some part of it from $u$ to the new, approximating minimizers. Unfortunately, in the vectorial case no maximum principle applies in general and therefore we employ a novel approximation using additional nonlocal terms. For this, we add a suitable truncated Gagliardo-type seminorm term to our functional, together with a more standard $L^{2d}$ penalization term. Specifically, we consider perturbed functionals of the type 
\begin{eqnarray}\label{nonloc}
w&\mapsto&\mathcal{F}(w,B)+\varepsilon\int_{B}\snr{Dw}^{2d} \, dx+\int_{B}(\snr{w}^{2}-M_0^{2})_{+}^{d} \, dx\nonumber \\
&& \hspace{13mm} \,  +\int_{\mathbb{R}^{n}}\int_{\mathbb{R}^{n}}\frac{\left(\snr{w(x)-w(y)}^{2}-M^{2}\snr{x-y}^{2\gamma}\right)^{d}_{+}}{\snr{x-y}^{n+2sd}} \, dx\dy
\end{eqnarray}
for $\eps$ small, and suitably large $d, M_0$ and $M$. The last two ones depend on $\|u\|_{C^{0, \gamma}}$, which is finite by assumption. Here, following a standard notation, we are denoting 
\eqn{tagli}
$$(t-k)_{+}:=\max\{t-k, 0\}\,, \qquad \mbox{for $t, k \in 
\er$}\,.$$ 
Functionals of mixed local/nonlocal type, in the quadratic/linear case $F(x, z)\equiv |z|^2$, have been recently studied in \cite{BDVV1, BDVV2} under special boundary conditions. As fas as we know, this is the first paper where nonlinear problems of this type are considered, and a priori estimates are presented. Let us mention that in the setting of functionals with $(p,q)$-growth conditions, purely local approximations aimed at preserving the starting $L^{\infty}$-information, have been considered for the first time in \cite{CKP} in the autonomous case; see also \cite{DM} for the nonautonomous case under assumptions \rif{assF}. These approximations are made using only the first line in \rif{nonloc}, and are not suitable to preserve the initial $C^{0, \gamma}$-regularity information. Adding the last line in \rif{nonloc} therefore turns out to be crucial. The second ingredient, which is taken from \cite{DM, ELM}, is a suitable use of the difference quotients techniques in the setting of Fractional Sobolev spaces. In this setting, the H\"older continuity of $\partial_zF(\cdot)$ in \rif{assF}$_3$ is automatically read as a fractional differentiability and allows to get estimates in Nikolski spaces for $Du$. At this stage, we finally use the last and third ingredient. In order to get advantage of the assumed H\"older continuity of $u$, we transfer this information to the approximating minimizers in terms of suitable Sobolev-Slobodevsky spaces. This allows to use a Gagliardo-Nirenberg type interpolation inequality, which improves the exponents intervening when using fractional Sobolev embedding theorem in the setting of Caccioppoli type inequalities; see Lemma \ref{fraclem} below. 
\section{Preliminaries}\label{pre}
\subsection{Notation}\label{notsec}
In the rest of the paper, we denote by $\Omega\subset \er^n$ an open subset, $n \geq 2$. 
We denote by $c$ a general constant larger than one. Different occurrences from line to line will be still denoted by $c$. Special occurrences will be denoted by $c_1, c_2,  \tilde c$ or likewise. When a relevant dependence on parameters occurs, this will be emphasized by putting the correspondent parameters in parentheses. Finally, the symbol $\lesssim$ denotes inequalities where absolute constants are involved. As usual, we denote by $ B_r(x_0):= \{x \in \er^n  :   |x-x_0|< r\}$, the open ball with center $x_0$ and radius $r>0$; when it is clear from the context, we omit denoting the center, i.e., $B_r \equiv B_r(x_0)$. When not otherwise stated, different balls in the same context will share the same center. We shall also denote $B_1 = B_1(0)$ if not differently specified. Finally, with $B$ being a given ball with radius $r$ and $\sigma$ being a positive number, we denote by $\sigma B$ the concentric ball with radius $\sigma r$. In denoting several function spaces like $L^p(\Omega), W^{1,p}(\Omega)$, we shall denote the vector valued version by $L^p(\Omega,\er^k), W^{1,p}(\Omega,\er^k)$ in the case the maps considered take values in $\er^k$, $k\in \en$. Sometimes we shall abbreviate $L^p(\Omega,\er^k)\equiv L^p(\Omega), W^{1,p}(\Omega,\er^k)\equiv W^{1,p}(\Omega)$. 
With $\mathcal B \subset \er^{n}$ being a measurable subset with bounded positive measure $0<|\mathcal B|<\infty$, and with $w \colon \mathcal B \to \er^{k}$, being a measurable map, we shall denote the integral average of $w$ over $\mathcal B$ by  
$$
   (w)_{\mathcal B} \equiv \mint_{\mathcal B}  w(x) \, dx  :=  \frac{1}{|\mathcal B|}\int_{\mathcal B}  w(x) \, dx\,.
$$
Given $z, \xi \in \er^{N\times n}$, their Frobenius product is defined as $z \cdot \xi= z_i^\alpha \xi_i^\alpha$; it follows that $  \xi \cdot \xi = |\xi|^2$ and in the rest of the paper we shall use the classical Frobenius norm for matrixes.  We shall use a similar notation for the scalar product in $\er^N$. As usual, the symbol $\otimes$ denotes the tensor product; in particular, given $\lambda \in \er^n$ and $\iota \in \er^n$, we have $\iota \otimes \lambda \equiv \{\iota^\alpha \lambda_i\} \in \er^{N\times n}$, $1\leq i \leq n$, $1\leq \alpha \leq N$.  In this paper we use the standard notation
\eqn{notationh}
$$
[w]_{0,\gamma; \mathcal B} := \sup_{x, y\in  \mathcal B, x\not=y}\, \frac{|w(x)-w(y)|}{|x-y|^\gamma}\,,
$$
whenever $ \mathcal B \subset \er^n$ is a subset, $\gamma \in (0,1]$ and $w \colon  \mathcal B \to \er^k$. Accordingly, the $C^{0,\gamma}$-norm of $w$ is defined by $\|w\|_{C^{0,\gamma}( \mathcal B)}:= \|w\|_{L^{\infty}(A)}+[w]_{0,\gamma; \mathcal B} $. 
 
\subsection{Fractional spaces and interpolation inequalities}
We collect here some basic facts about fractional Sobolev spaces. We refer to \cite{ELM, pala} for more results. For a map $w \colon \Omega \to \er^k$ and a vector $h \in \er^n$, we denote by 
$\tau_{h}\colon L^1(\Omega,\er^k) \to L^{1}(\Omega_{|h|},\er^k)$ the standard finite difference
operator pointwise defined as
\eqn{ttau1}
$$
(\tau_{h}w)(x)\equiv \tau_{h}w(x):=w(x+h)-w(x)\;,
$$
whenever $\Omega_{|h|}:=\{x \in \Omega \, : \, 
\dist(x, \partial \Omega) > |h|\}$ is not empty. We shall consider a similar operator acting on maps $\psi \colon  \Omega\times \Omega \to \er^k$, this time defined by 
\eqn{ttau2}
$$ 
(\tilde \tau_{h}w)(x)\equiv\tilde \tau_{h}\psi(x,y):=\psi(x+h, y+h)-\psi(x,y)\,.
$$
\begin{definition}\label{fra1def}
Let $\alpha_{0} \in (0, \infty)\setminus \en$, $p \in [1, \infty)$, $k \in \en$, $n \geq 2$, and let $\Omega \subset \er^n$ be an open subset.
\begin{itemize}
\item If $\alpha_{0} \in (0,1)$, the fractional Sobolev space $W^{\alpha_{0} ,p}(\Omega,\er^k )$ consists of those maps $w \colon \Omega \to \er^k$ such that 
the following Gagliardo type norm is finite:
\begin{eqnarray}
\notag
\| w \|_{W^{\alpha_{0} ,p}(\Omega )} & := &\|w\|_{L^p(\Omega)}+ \left(\int_{\Omega} \int_{\Omega}  
\frac{|w(x)
- w(y) |^{p}}{|x-y|^{n+\alpha_{0} p}} \, dx \dy \right)^{1/p}\\
&=:& \|w\|_{L^p(\Omega)} + [w]_{\alpha_{0}, p;\Omega}\,.\label{gaglia}
\end{eqnarray}
In the case $\alpha_{0} = [\alpha_{0}]+\{\alpha_{0}\}\in \en + (0,1)>1$, it is $w\in W^{\alpha_{0} ,p}(\Omega,\er^k )$ iff  $$
\| w \|_{W^{\alpha_{0} ,p}(\Omega )}  := \| w \|_{W^{[\alpha_{0}],p}(\Omega)} +[D^{[\alpha_{0}]}w]_{\{\alpha_{0}\}, p;\Omega}$$ is finite. The local variant $W^{\alpha_{0} ,p}_{\loc}(\Omega,\er^k )$ is defined by requiring that $w \in W^{\alpha_{0} ,p}_{\loc}(\Omega,\er^k )$ iff $w \in W^{\alpha_{0} ,p}(\tilde{\Omega},\er^k)$ for every open subset $\tilde{\Omega} \Subset \Omega$. 
\item For $\alpha_{0} \in (0,1]$, the Nikol'skii space $N^{\alpha_{0},p}(\Omega,\er^k )$ is defined by prescribing that $w \in N^{\alpha_{0},p}(\Omega,\er^k )$ if and only if
$$\| w \|_{N^{\alpha_{0},p}(\Omega,\er^k )} :=\|w\|_{L^p(\Omega,\er^k)} + \left(\sup_{|h|\not=0}\, \int_{\Omega_{|h|}} 
\frac{|w(x+h)
- w(x) |^{p}}{|h|^{\alpha_{0} p}} \, dx  \right)^{1/p}\;.$$
The local variant $N^{\alpha_{0},p}_{\loc}(\Omega,\er^k )$ is defined analogously to $W^{\alpha_{0} ,p}_{\loc}(\Omega,\er^k )$.\end{itemize}
\end{definition}
We have that $ W^{\alpha_{0} ,p}(\Omega,\er^k)\subsetneqq N^{\alpha_{0},p}(\Omega,\er^k)\subsetneqq
W^{\beta,p}(\Omega,\er^k)$, for every $\beta <\alpha_{0}$, hold for sufficiently domains $\Omega$. These inclusions are somehow quantified in the following lemma 
\begin{lemma} \label{l2}
Let $w\in L^p (\Omega )$, $p\geq 1$, and assume that for $\alpha_0 \in (0,1]$, $S \geq 0$ and an open and bounded set $\tilde{\Omega}
\Subset\Omega$ we have that 
$\|\tau_{h}w\|_{
L^p(\tilde{\Omega})} \leq S |h|^{\alpha_0}$ holds for every
 $h\in \er^n$ satisfying $0<|h|\leq d$, where
$0< d \leq \dist (\tilde{\Omega},
\partial \Omega )$. Then $w\in W^{\beta
,p}(\tilde \Omega,\er^k )$ for every $\beta \in
(0,\alpha_0)$, and the estimate
\eqn{eess}
$$
\| w\|_{W^{\beta ,p}(\tilde \Omega)}\leq c\left(\frac{d^{(\alpha_0-\beta)}S}{(\alpha_0 -\beta)^{1/p}}
+ \frac{\| w\|_{L^{p}(\tilde \Omega)}}{\min\{d^{n/p+\beta},1\}}\right)
$$
holds with $c\equiv c (n,p)$. In particular, let $B_{\varrho} \Subset B_{r}\subset \er^n$ be concentric balls with $r\leq 1$, $w\in L^{p}(B_{r},\mathbb{R}^{k})$, $p>1$ and assume that, for $\alpha_{0} \in (0,1]$, $S\ge 1$, there holds
\eqn{cru1}
$$
\nr{\tau_{h}w}_{L^{p}(B_{\rr},\er^k)}\le S\snr{h}^{\alpha_{0} } \quad \mbox{
for every $h\in \mathbb{R}^{n}$ with $0<\snr{h}\le \frac{r-\rr}{K}$, where $K \geq 1$}\;.
$$
Then it holds that 
\eqn{cru2}
$$
\nr{w}_{W^{\beta,p}(B_{\rr},\er^k)}\le\frac{c}{(\alpha_{0} -\beta)^{1/p}}
\left(\frac{r-\rr}{K}\right)^{\alpha_{0} -\beta}S+c\left(\frac{K}{r-\rr}\right)^{n/p+\beta} \nr{w}_{L^{p}(B_{r},\er^k)}\,,
$$
where $c\equiv c(n,p)$. 
\end{lemma}\begin{proof}  
A main point in Lemma \ref{l2}, is the precise quantitative linkage between the size of $|h|$ appearing in \rif{cru1}, and the dependence on the constants appearing in \rif{cru2}. This will be crucial in the applications we shall made of it. See Sections \ref{hsec1} and \ref{hsec2} below. For this reason we decide to report the easy proof, since it does not appear in the literature but is often reported as folklore. 
Fubini's theorem yields
\begin{eqnarray}
&&\nonumber \int_{\tilde \Omega} \int_{\tilde \Omega\cap \{|x-y|<d\}}  
\frac{|w(x)
- w(y) |^{p}}{|x-y|^{n+\beta p}} \ dx \, dy \leq \int_{B_d(0)} \int_{\tilde \Omega}  
\frac{|w(x+h)
- w(x) |^{p}}{|h|^{n+\beta p}} \ dx \, dh \\
\nonumber
&& \qquad \qquad \qquad\qquad   \leq  \sup_{0<|h|\leq d}\, \int_{\tilde \Omega} 
\frac{|w(x+h)
- w(x) |^{p}}{|h|^{\alpha_0 p}} \ dx   \int_{B_d(0)}\frac{dh}{|h|^{n+(\beta-\alpha_0)p}} \\
\nonumber
&&\qquad \qquad \qquad\qquad \qquad \qquad  \leq  cS^p\int_{0}^{d}\frac{dt}{t^{1+(\beta-\alpha_0)p}} \leq \frac{cd^{(\alpha_0-\beta)p}S^p}{(\alpha_0-\beta)p}
\nonumber \;.
\end{eqnarray}
On the other hand we have that 
\begin{flalign*}
 \int_{\tilde \Omega} \int_{\tilde \Omega\cap \{|x-y|\geq d\}}  
\frac{|w(x)
- w(y) |^{p}}{|x-y|^{n+\beta p}} \ dx \, dy  &\leq \frac{2^{p-1}}{d^{n+\beta p}}\int_{\tilde \Omega} \int_{\tilde \Omega}  
(|w(x)|^p +|w(y)|^{p}) \ dx \, dy\\& = \frac{2^{p}\|w\|_{L^p(\tilde \Omega)}^p}{d^{n+\beta p}}\,.
\end{flalign*}
Connecting the estimates in the last two displays yields
$$
\int_{\tilde \Omega} \int_{\tilde \Omega}  
\frac{|w(x)
- w(y) |^{p}}{|x-y|^{n+\beta p}} \ dx \, dy \leq c \left(\frac{d^{(\alpha_0-\beta)q}S^p}{\alpha_0 -\beta}+ \frac{\| w\|_{L^{p}(\tilde \Omega)}^p}{d^{n+\beta p}}\right)
$$ 
with $c\equiv c(n)$, from which the full inequality in \rif{eess} immediately follows. 
\end{proof}
Recalling notations \rif{notationh} and \rif{gaglia}, we now report a suitable version of the fractonal Sobolev-Morrey embedding. 
\begin{lemma}\label{l9}
Let $w\in W^{s,t}(\er^n,\mathbb{R}^{k})$, with $t\ge 1$, $s\in (0,1)$ such that $s t>n$. If $B\subset \er^n$ is a ball, then $ w\in C^{0,s-n/t}(B,\mathbb{R}^{k})$ and the inequality 
\eqn{prosca}
$$
[w]_{0,s-n/t;B} \leq c  [w]_{s, t;\er^n}, 
$$
holds for a constant $c$, depending only on $n,s,t$.
\end{lemma}
\begin{proof} This easily follows from the standard proofs in the literature; just a comment on the dependence of the constant $c$. First observe that one can reduce the proof of \rif{prosca} to the case when $B\equiv B_1(0)$ via a standard scaling argument; i.e., considering 
$\tilde u(x):= R^{-s+n/t}u(x_0+Rx)$, where $B= B_R(x_0)$. Once the scaling is done, we can use \cite[(8.8)]{pala} and the fractional Poincar\'e inequality to get \rif{prosca} with an absolute constant depending only on $n,s,t$. 
\end{proof}
The following is a Gagliardo-Nirenberg type inequality in fractional Sobolev spaces, taken from \cite[Section 2.3]{CM2}. The proof relies on a localization argument combined with the global inequalities from \cite[Lemma 1, pag. 329]{BrMi}; see also \cite{BM2, mazya}. To get the explicit dependence of the constants by the radii $1/(r-\varrho)^\kappa$ below, it is sufficient to trace back the dependence on the various constants in \cite[Lemma 2.6]{CM2}. 
\begin{lemma}\label{fraclem}
Let $B_{\rr}\Subset B_{r}\Subset \er^n$ be concentric balls with $r \leq 1$. Let $0\le s_{1}<1<s_{2}<2$, $1<a,t<\infty$, $\tilde{p}>1$ and $\theta\in (0,1)$ be such that
\begin{flalign*}
1=\theta s_{1}+(1-\theta)s_{2},\qquad \frac{1}{\tilde{p}}=\frac{\theta}{a}+\frac{1-\theta}{t}.
\end{flalign*}
Then every function $w\in W^{s_{1},a}(B_{r})\cap W^{s_{2},t}(B_{r})$ belongs to $W^{1,\tilde{p}}(B_{\rr})$ and the inequality 
\eqn{viavia}
$$
\nr{Dw}_{L^{\tilde{p}}(B_{\rr})}\le \frac{c}{(r-\rr)^{\kappa}}[w]^{\theta}_{s_{1},a;B_{r}}\nr{Dw}^{1-\theta}_{W^{s_{2}-1,t}(B_{r})}
$$
holds for constants $c,\kappa\equiv c,\kappa(n,s_{1},s_{2},a,t)$.
\end{lemma}
Next, a classical iteration lemma of \cite[Lemma 6.1]{giu}, that is 
\begin{lemma}\label{itercz}
Let $\mathcal{Z}\colon [\varrho_{0},\varrho_{1}]\to \mathbb{R}$ be a nonnegative and bounded function, and let $\theta \in (0,1)$ and $A,B\ge 0$, $\gamma_{1},\gamma_{2}\ge 0$ be numbers. Assume that
\begin{flalign*}
\mathcal{Z}(t)\le \theta \mathcal{Z}(s)+\frac{A}{(s-t)^{\gamma_{1}}}+\frac{B}{(s-t)^{\gamma_{2}}}
\end{flalign*}
holds for $\varrho_{0}\le t<s\le \varrho_{1}$. Then the following inequality holds with $c\equiv c(\theta,\gamma_{1},\gamma_{2})$:
\begin{flalign*}
\mathcal{Z}(\varrho_{0})\le \frac{cA}{(\varrho_{1}-\varrho_{0})^{\gamma_{1}}}+\frac{cB}{(\varrho_{1}-\varrho_{0})^{\gamma_{2}}}\;.
\end{flalign*}
\end{lemma}
\section{Proof of Theorems \ref{t1},\ref{t2} and Corollaries \ref{c0},\ref{c1}}
These proofs take twelve different steps. The first ten are dedicated to the proof of Theorem \ref{t1}. In Step 11 we deal with Corollaries \ref{c0},\ref{c1} and Step 12 is dedicated to Theorem \ref{t2}. 
\subsection{Step 1: Choice of parameters}\label{st1} Fix $\mathfrak{q}$ as in \rif{tip}. We start considering parameters $s,d$ and $\beta$ initially satisfying
\eqn{scelta iniziale}
$$ 0 \leq s < \gamma  \,, \qquad 2d > \max\{q,n\}\,,\ \qquad  0< \beta < \min\{\alpha,2\gamma\}\,.$$
Accordingly, we define the function $\tilde p \equiv \tilde p (s,d, \beta)$ as
\begin{flalign}\label{tp}
\tilde{p}:=\frac{2d\left[p(1-s)+\beta\right]}{\beta+2d(1-s)}\ \ \mbox{if}\ \ p\ge 2\qquad \mbox{and}\qquad \tilde{p}:=\frac{2dp[2(1-s)+\beta]}{p\beta+4d(1-s)} \ \ \mbox{if} \ \ 1<p<2\,.
\end{flalign}
Notice that the inequality $d > p/2$, which is true by \rif{scelta iniziale}, in particular implies 
\eqn{luckybound}
$$
\tilde p < 2d\,.
$$
The function $\tilde p$ is increasing in all its variables (we again use that $d \geq p/2$ for this), with
\eqn{vava1}
$$
\lim_{s\to \gamma, d\to \infty, \beta\to \min\{\alpha,2\gamma\}}\, \tilde p (s,d, \beta)= 
p+\frac{\min\{\alpha,2\gamma\}}{\vartheta(1-\gamma)}\,,$$
where $\vartheta$ has been defined in \rif{pq}. In the following, we always take $d$ such that
\eqn{beta0}
 $$d > \frac{\max\{q,n\}}{2s} \Longrightarrow \beta_0:=s-\frac{n}{2d}>0\,.$$
In view of \rif{vava1} we further increase both $s$ and $d$, while still keeping \rif{beta0}, and find $\beta$  such that 
\eqn{beta1}
$$
\beta < \alpha_0:=\min\{\alpha, 2\beta_0\}< \min\{\alpha, 2\gamma\}
$$
and 
\eqn{range}
$$
 q \leq  \mathfrak{q}< \tilde p (s,d, \beta)  <  p+ \frac{\min\{\alpha,2\gamma\}}{ \vartheta(1-\gamma)}\;.
$$
Notice that, by further increasing $s,d,\beta$ still in the range fixed in \rif{scelta iniziale}, conditions \rif{beta0}-\rif{range} still hold (as noticed before, $\tilde p(s,d,\beta)$ is increasing with respect to all its variables). Keeping this in mind, we next distinguish two cases. When $p\geq 2$, notice that 
\begin{flalign*}
&\lim_{s\to \gamma, d\to \infty, \beta \to  \min\{\alpha,2\gamma\}}\frac{\tilde{p}(s,d,\beta)(q-p)}{\tilde{p}(s,d,\beta)-p}\frac{1-s}{p(1-s)+\beta} \\ & \qquad \qquad = 
  \lim_{s\to \gamma, d\to \infty, \beta \to  \min\{\alpha,2\gamma\}} \frac{2d(q-p)(1-s)}{\beta (2d-p)} =
\frac{(q-p)(1-\gamma)}{\min\{\alpha,2\gamma\}} \stackrel{\rif{pq}}{<}1\,.
\end{flalign*}
Therefore, we finally again increase $s,d$ and $\beta$ again, in order to have
\eqn{270}
$$
\frac{\tilde{p}(s,d,\beta)(q-p)}{\tilde{p}(s,d,\beta)-p}\frac{1-s}{p(1-s)+\beta}<1\,.
$$
When instead $1<p<2$, we notice that 
\begin{flalign*}
&\lim_{s\to \gamma, d\to \infty, \beta \to  \min\{\alpha,2\gamma\}}\frac{\tilde{p}(s,d,\beta)(q-p)}{\tilde{p}(s,d,\beta)-p}\frac{2(1-s)}{p[2(1-s)+\beta]}\\ & \qquad \qquad 
\lim_{s\to \gamma, d\to \infty, \beta \to  \min\{\alpha,2\gamma\}} \frac{4d(q-p)(1-s)}{p\beta(2d-p)} =  \frac{2(q-p)(1-\gamma)}{p\min\{\alpha,2\gamma\}} \stackrel{\rif{pq}}{<}1\,,
\end{flalign*}
and therefore we again find $s, d$ and $\beta$ such that 
\eqn{280}
$$
\frac{\tilde{p}(s,d,\beta)(q-p)}{\tilde{p}(s,d,\beta)-p}\frac{2(1-s)}{p[2(1-s)+\beta]}<1\,.
$$
From now on we shall always consider this final choice of $s,d$ and $\beta$, and therefore we shall use \rif{tp}-\rif{280} for the rest of the proof. In view of this, from now on we shall express any dependency from $(s,d,\beta)$ as a dependence on $(p,q,\alpha,\gamma, \mathfrak q)$.
In the following, in order to shorten the notation, we shall denote
\eqn{idati}
$$
\data\equiv (n,N,p,q, \nu, L, \alpha,\gamma, \mathfrak q) \,, \qquad \data_{\rm e}  \equiv (n,p,q,\alpha,\gamma, \mathfrak q)\,.
$$
\subsection{Step 2: H\"older and Sobolev extensions}\label{st2} Let us notice that in the definition \rif{lav0} we can replace $\mathcal{C}_H(w,B)$ in \rif{ilC} by
$$
\left\{\{w_{j}\}\subset W^{1,\infty}(B,\mathbb{R}^{N})\colon (w_j)_{B}=0, \ w_{j}\rightharpoonup w \ \mbox{in} \ W^{1,p}(B,\mathbb{R}^{N}), \ \sup_{j}\, [w_{j}]_{0,\gamma;B} \leq H \right\}.
$$  
Such a replacement leaves the values of $\bar{\mathcal{F}}_H$ and $\mathcal {L}_{F,H}$ unaltered. Keeping this fact in mind, we fix a ball $B_{r}\subset B_{2r}\Subset \Omega$ with $r\le 1$ and such that \eqref{nolav} holds, and, by Proposition \ref{approssi1}, we can find a sequence 
\eqn{apreg}
$$\{\tilde{u}_{j}\}\subset W^{1,\infty}(B_{r},\mathbb{R}^{N})\cap C^{0,\gamma}(B_{r},\mathbb{R}^{N}),\qquad (\tilde u_j)_{B}=0$$ such that, eventually passing to a not relabelled subsequence, it holds that
\begin{flalign}\label{conv}
\begin{cases}
\tilde{u}_{j}\rightharpoonup u \ \mbox{weakly in} \ W^{1,p}(B_{r}, \er^N), \ \tilde u_j \to u  \ \mbox{strongly in} \ L^{p}(B_{r}, \er^N) \\
 \mathcal{F}(\tilde{u}_{j},B_{r})\to \mathcal{F}(u,B_{r})\\
r^{-\gamma}\|\tilde{u}_{j}\|_{L^\infty(B_{r})}+ [\tilde{u}_{j}]_{0,\gamma;B_{r}}+r^{-\gamma -n/p}\|\tilde{u}_{j}\|_{L^p(B_{r})} \le cH\\
\nu \|D\tilde{u}_{j}\|_{L^p(B_{r})}^p\leq  \mathcal{F}(\tilde{u}_{j},B_{r})\leq   \mathcal{F}(u,B_{r})+1 
\end{cases}
\end{flalign}
with the last two lines that hold for every $j \geq 1$, and where $c\equiv c(n,p)$ is an absolute constant. All the facts from \rif{apreg} are directly coming from Proposition \ref{approssi1} but \rif{conv}$_3$, that maybe deserves a few words; there only $ [\tilde{u}_{j}]_{0,\gamma;B_{r}}\leq H$ directly comes from the definition of the Lavrentiev gap. For this, notice that, if $x \in B_r$, then by \rif{apreg} we find, by Jensen's inequality, that 
$
|\tilde u_j(x)| = |\tilde u_j(x)- (\tilde u_j)_{B}|\leq cr^\gamma [\tilde u_{j}]_{0,\gamma;B_r}\leq cr^\gamma H. 
$ This implies \rif{conv}$_3$. Now, by \cite[Theorem 2]{mi}, we can extend the $\tilde{u}_{j}$'s to the whole $\mathbb{R}^{n}$ by determining maps $\{\bar{u}_{j}\}$ such that
\begin{flalign}\label{3bis}
[\bar{u}_{j}]_{0,\gamma;\er^n}\le  c[\tilde{u}_{j}]_{0,\gamma;B_{r}} \leq cH
\end{flalign}
for an absolute constant $c$ which is independent of $j$. 
Let $\bar{\eta}\in C^{1}_{0}(B_{3r/2})$ be such that
\begin{flalign}\label{bbeta}
\mathds{1}_{\bar{B}_{r}}\le \bar{\eta}\le \mathds{1}_{\bar{B}_{3r/2}}\quad\mbox{and}\quad  \snr{D\bar{\eta}}\lesssim 1/r
\end{flalign}
and set $\bar{v}_{j}:=\bar{u}_{j}\eta$. By this very definition, \eqref{3bis} and \eqref{bbeta} we have that
\begin{flalign}\label{3.01}
\begin{cases}
\ r^{-\gamma}\|\bar{v}_{j}\|_{L^{\infty}(\er^n)} \leq c_* H\\
\ [\bar{v}_{j}]_{0,\gamma;\mathbb{R}^{n}}\leq  c_*H\\
\ \bar{v}_{j}\equiv \bar{u}_{j}\equiv \tilde{u}_{j} \ \ \mbox{in} \ \ B_{r}\\
\ \supp\, \bar{v}_{j}\Subset B_{3r/2}\\
\   \bar{v}_{j} \in W^{1,2d}(B_r,\er^N) \\
\  r^{s} [\bar v_j]_{s,2d;\er^n}\leq c_*r^{\gamma+n/(2d)}H
\end{cases}
\end{flalign}
hold for every $j\in \en$, an absolute constant $c_*\equiv c_* (n,d,\gamma,s)
\equiv c_* ( \data_{\rm e})$, which is independent of $j$. We confine ourselves to sketch the simple proofs of \rif{3.01}$_{1,2}$ and of \rif{3.01}$_{6}$, the other assertions  being a direct consequence of the definition of $\bar v_j$ (recall also \rif{apreg}). Let us first remark that, triangle inequality, \rif{3bis} and finally \rif{conv}$_3$, imply
\eqn{vedetto}
$$
\|\bar u_j\|_{L^\infty(B_{2r})} \leq c[\tilde u_j]_{0,\gamma;B_{r}}r^\gamma + \|\tilde u_j\|_{L^\infty(B_{r})} \leq c Hr^{\gamma}\,,
$$
that is \rif{3.01}$_1$. In order to prove \rif{3.01}$_{2}$, it is sufficient to notice that if $x,y \in B_{3r/2}$,  then, using \rif{3bis} and \rif{vedetto} it follows 
\begin{flalign*}
|\bar{v}_{j}(y)-\bar{v}_{j}(x)| & =  
|\bar{u}_{j}(y)\eta(y)-\bar{u}_{j}(x)\eta(x)| \leq  
|\bar{u}_{j}(y)-\bar{u}_{j}(x)| +| \eta(y)-\eta(x)|\|\bar{u}_{j}\|_{L^\infty(B_{2r})}\\
& \leq c[\bar u_j]_{0,\gamma;B_{2r}}|x-y|^{\gamma}+ \|D \eta\|_{L^{\infty}(B_r)} \|\bar{u}_{j}\|_{L^\infty(B_{2r})}|x-y|\\
& \leq cH|x-y|^{\gamma}+\frac{c}{r^\gamma}\|\bar{u}_{j}\|_{L^\infty(B_{2r})}|x-y|^{\gamma} \leq 
c_*H|x-y|^{\gamma}\,.
\end{flalign*}
As it is $v_j\equiv 0$ outside $B_{3r/2}$, this is sufficient to conclude with \rif{3.01}$_{2}$. Finally, the proof of \rif{3.01}$_{6}$. By \rif{3.01}$_{1,2,3,4}$, we have 
\begin{flalign}\label{20.02pre}
\notag [\bar v_j]_{s,2d;\er^n}^{2d}&=2\int_{\mathbb{R}^{n}\setminus B_{2r}}\int_{B_{2r}}\frac{\snr{\bar{v}_{j}(x)-\bar{v}_{j}(y)}^{2d}}{\snr{x-y}^{n+2sd}} \dx\dy +\int_{B_{2r}}\int_{B_{2r}}\frac{\snr{\bar{v}_{j}(x)-\bar{v}_{j}(y)}^{2d}}{\snr{x-y}^{n+2sd}} \dx\dy\nonumber \\
&\,  =2\int_{\mathbb{R}^{n}\setminus B_{2r}}\int_{B_{3r/2}}\frac{\snr{\bar{v}_{j}(x)}^{2d}}{\snr{x-y}^{n+2sd}} \dx\dy +\int_{B_{2r}}\int_{B_{2r}}\frac{\snr{\bar{v}_{j}(x)-\bar{v}_{j}(y)}^{2d}}{\snr{x-y}^{n+2sd}} \dx\dy\nonumber \\
&\, \le c\int_{\mathbb{R}^{n}\setminus B_{r/2}}\frac{\dz}{\snr{z}^{n+2sd}} \, \|\bar{v}_{j}\|_{L^{2d}(B_{3r/2})}^{2d} +c [\bar{v}_j]_{0, \gamma;B_{2r}}^{2d}
\int_{B_{2r}}\int_{B_{2r}}\frac{\dx\dy}{\snr{x-y}^{n+2d(s-\gamma)}} 
\nonumber \\
&\, \le \frac{cr^{n+2d(\gamma-s)}H^{2d}}{s(\gamma-s)} \equiv c r^{n+2d(\gamma-s)}H^{2d}
\end{flalign}
for $c\equiv c (\data_{\rm e})$, 
and the proof of \rif{3.01}$_{6}$ follows. 
\subsection{Step 3: Approximation via nonlocal functionals}
We introduce the nonlocal Dirichlet class
$$\mathbb X( \bar{v}_{j},B_r):=\left\{v \in \left( \bar{v}_{j}+W^{1,2d}_{0}(B_{r},\mathbb{R}^{N})\right)\cap W^{s,2d}(\mathbb{R}^{n},\mathbb{R}^{N})\, \colon\, v \equiv \bar{v}_{j}\  \mbox{on}\  \er^n\setminus B_r\right\}\,.$$ This is a convex,  closed subset of $W^{1,2d}_{0}(B_{r},\mathbb{R}^{N})\cap W^{s,2d}(\mathbb{R}^{n},\mathbb{R}^{N})$, and it is non-empty, as $ \bar{v}_{j} \in \mathbb X( \bar{v}_{j},B_r)$ by \rif{3.01}. Next, we define $u_{j}\in \mathbb X( \bar{v}_{j},B_r)$ as the solution to
\begin{flalign}\label{pd}
u_{j}\mapsto \min_{w\in \mathbb X( \bar{v}_{j},B_r)}\,  \mathcal{F}_{j}(w,B_{r})\,,
\end{flalign}
where, keeping in mind the notation in \rif{tagli}, it is
\begin{flalign}
 \mathcal{F}_{j}(w,B_{r})&:=\mathcal{F}(w,B_{r})+\varepsilon_{j} \int_{B_{r}}(\snr{Dw}^{2}+\mu^{2})^{d} \, dx\nonumber \\
&\quad \ \  +\int_{B_{r}}(\snr{w}^{2}-M_0^{2})_{+}^{d} \, dx+\int_{\mathbb{R}^{n}}\int_{\mathbb{R}^{n}}\frac{\left(\snr{w(x)-w(y)}^{2}-M^{2}\snr{x-y}^{2\gamma}\right)^{d}_{+}}{\snr{x-y}^{n+2sd}} \, dx\dy\label{ifunzionali}
\end{flalign}
with
\eqn{1}
$$
\left\{
\begin{array}{c}
\displaystyle  \varepsilon_{j}:=\frac{1}{\left(\nr{D\bar{v}_{j}}^{4d}_{L^{2d}(B_{r})}+j+1\right)}\\ [20 pt] 
\displaystyle
 M_0:=16c_*r^{\gamma}H , \quad  M:=16c_*H\,.
\end{array}
\right.
$$
In \rif{1} $c_*\equiv c_*( \data_{\rm e})$ is the same (absolute) constant has been defined in \rif{3.01}. Let us briefly point out how Direct Methods of the Calculus of Variations apply here to get the existence of $u_j$ in \rif{pd}; the main point is essentially to prove that the functional is coercive on $\mathbb X( \bar{v}_{j},B_r)$. With $j\in \en$ being fixed, let us consider a minimizing sequence $\{w_{j,k}\}_k \subset \mathbb X( \bar{v}_{j},B_r)$, i.e., such that 
\eqn{infy}
$$
\lim_{k}\, \mathcal{F}_{j}(w_{j,k},B_{r})=  \inf_{w\in \mathbb X( \bar{v}_{j},B_r)}\,  \mathcal{F}_{j}(w,B_{r})\,.$$ 
We now have (recall that $w_{j,k}\equiv \bar{v}_{j}$ outside $B_r$, $\bar v_j \equiv 0$ outside $B_{3r/2}$ and $ \bar{v}_{j}\equiv \tilde u_j$ in $B_r$)
\begin{flalign}
\notag \|w_{j,k}\|_{L^{2d}(\er^n)}^{2d} &\leq \|\bar v_{j}\|_{L^{2d}(B_{3r/2})}^{2d} +
\|w_{j,k}\|_{L^{2d}(B_r)}^{2d}  \leq 
cr^{n+2d\gamma} H^{2d} + \|w_{j,k}-\bar v_j\|_{L^{2d}(B_r)}^{2d}
 \\
\notag& \leq cr^{n+2d\gamma} H^{2d} +c r^{2d}  \|Dw_{j,k}\|_{L^{2d}(B_{r})}^{2d} +
cr^{2d}  \|D\bar v_{j}\|_{L^{2d}(B_{r})}^{2d}\\
&
\leq 
\frac {cr^{2d} \mathcal{F}_{j}(w_{j,k},B_{r})}{\eps_j } +cr^{2d}  \|D\tilde u_{j}\|_{L^{2d}(B_{r})}^{2d} + c r^{n+2d\gamma} H^{2d} 
\label{rerere}
\end{flalign}
for $c \equiv c(\data_{\rm e})$; in the second estimate in the above display we have used also \rif{3.01}$_1$ and in the last line the very definition of $\mathcal F_j$ from \rif{ifunzionali}. Furthermore, triangle inequality implies 
\begin{flalign}\label{esatto1}
\notag[w_{j,k}]_{s, 2d; B_{2r}}^{2d}&\leq c\int_{B_{2r}}\int_{B_{2r}}\frac{(\snr{w_{j,k}(x)-w_{j,k}(y)}^{2}-M^{2}\snr{x-y}^{2\gamma})^{d}_{+}}{\snr{x-y}^{n+2sd}} \dx\dy\nonumber \\
&\qquad \ + cM^{2d}\int_{B_{2r}}\int_{B_{2r}}\frac{\dx\dy}{\snr{x-y}^{n-2d(\gamma-s)}}\nonumber \\
&\leq  c\mathcal{F}_{j}(w_{j,k},B_{r})+ \frac{cM^{2d}r^{n+2d(\gamma-s)}}{\gamma-s}
=c\mathcal{F}_{j}(w_{j,k},B_{r})+ cr^{n+2d(\gamma-s)}H^{2d}\,,
\end{flalign}
where $c\equiv c( \data_{\rm e})$ and we have used the definition of $M$ from \eqref{1}$_2$ in the last line. 
In turn, using $\eqref{3.01}_1$, and \eqref{rerere}-\eqref{esatto1}, we have
\begin{flalign}
\notag [w_{j,k}]_{s, 2d; \er^n}^{2d}&  = 2\int_{\mathbb{R}^{n}\setminus B_{2r} }\int_{B_{2r}}\frac{\snr{w_{j,k}(x)-w_{j,k}(y)}^{2d}}{\snr{x-y}^{n+2sd}} \dx\dy+[w_{j,k}]_{s, 2d; B_{2r}}^{2d}\nonumber \\
&\notag =2\int_{\mathbb{R}^{n}\setminus B_{2r}}\int_{B_{3r/2}\setminus B_{r}}\frac{\snr{\bar{v}_{j}(x)}^{2d}}{\snr{x-y}^{n+2sd}} \dx\dy+2\int_{\mathbb{R}^{n}\setminus B_{2r}}\int_{B_{r}}\frac{\snr{w_{j,k}(x)}^{2d}}{\snr{x-y}^{n+2sd}} \dx\dy\nonumber \\
&\notag  \qquad +c\mathcal{F}_{j}(w_{j,k},B_{r})+ cr^{n+2d(\gamma-s)}H^{2d}\nonumber \\
&\notag  \le c\int_{\mathbb{R}^{n}\setminus B_{r/2}}\frac{\dz}{\snr{z}^{n+2sd}}\, \|\bar v_{j}\|_{L^{2d}(B_{3r/2})}^{2d}
+\int_{\mathbb{R}^{n}\setminus B_{r/2}}\frac{\dz}{\snr{z}^{n+2sd}}\, 
\|w_{j,k}\|_{L^{2d}(B_r)}^{2d} 
\nonumber \\
&\notag  \qquad +c\mathcal{F}_{j}(w_{j,k},B_{r})+ cr^{n+2d(\gamma-s)}H^{2d}\nonumber \\
& \le \frac {c\mathcal{F}_{j}(w_{j,k},B_{r})}{\eps_j }
+cr^{2d(1-s)}\|D\tilde u_{j}\|_{L^{2d}(B_{r})}^{2d}+  cr^{n+2d(\gamma-s)}H^{2d}\label{splitas}
\end{flalign}
for $c\equiv c( \data_{\rm e})$. Using the content of \eqref{rerere}-\eqref{splitas}, we find (recall it is $r\leq 1$) 
$$
\|w_{j,k}\|_{W^{1,2d}(B_{r})} + \|w_{j,k}\|_{W^{s,2d}(\er^n)}  \leq   \frac {c[\mathcal{F}_{j}(w_{j,k},B_{r})]^{1/(2d)}}{\eps_j ^{1/(2d)}}
+c\|D\tilde u_{j}\|_{L^{2d}(B_{r})} +  cH\,.
$$
This last estimate and \rif{infy} imply that the sequence $\{w_{j,k}\}_k$ is bounded in $W^{1,2d}(B_{r},\mathbb{R}^{N})\cap W^{s,2d}(\mathbb{R}^{n},\mathbb{R}^{N})$ and therefore, up to a not relabelled subsequence, we can assume that $w_{j,k} \rightharpoonup u_j$ weakly as $k \to \infty$, both in $W^{1,2d}(B_{r},\mathbb{R}^{N})$ and $W^{s,2d}(\mathbb{R}^{n},\mathbb{R}^{N})$, for some $u_j \in \mathbb X( \bar{v}_{j},B_r)$. Moreover, again up to a diagonalization argument, we can also assume $w_{j,k}\to u_j$ a.e. At this point we use lower semicontinuity.  Specifically, we use the convexity of $z\mapsto F(x, z)$ and $z\mapsto |z|^{2d}$ to deal with the local part, and Fatou's lemma to deal with the nonlocal one, in order to get
$$
\mathcal{F}_{j}(u_{j},B_{r})\leq  \liminf_{k\to \infty}\mathcal{F}_{j}(w_{j,k},B_{r})\,,
 $$
thereby proving the minimality in \rif{pd} accordingly to the usual Direct Methods of the Calculus of Variations (see for instance \cite{giu}). 
\subsection{Step 4: Convergence to $u$} Here we prove that, up to a non relabelled subsequence, $\{u_j\}$ weakly converges to the original minimizer $u$ in $W^{1,p}(B_{r},\mathbb{R}^{N})$. For this, we start observing that 
$\eqref{1}_{1}$ guarantees that
\begin{flalign}\label{2}
 \varepsilon_{j}\int_{B_{r}}(\snr{D\bar{v}_{j}}^2+\mu^{2})^{d} \, dx \to 0\,.
\end{flalign}
Moreover, using \rif{3.01}$_{2}$ and the definition of $M$ in \rif{1}$_2$, we have that $
 \snr{\bar{v}_{j}(x)-\bar{v}_{j}(y)}\leq M\snr{x-y}^{\gamma}/4\,,
$
holds for every $j \in \en$ and $x, y \in \er^n$. In turn, this implies 
$$
\int_{\mathbb{R}^{n}}\int_{\mathbb{R}^{n}}\frac{(\snr{\bar{v}_{j}(x)-\bar{v}_{j}(y)}^{2}-M^{2}\snr{x-y}^{2\gamma})_{+}^{d}}{\snr{x-y}^{n+2sd}} \dx\dy=0
$$
for every $j \in \en$. By $\eqref{3.01}_{3}$ and the definition of $M_0$ in \rif{1}$_2$, it is
$$
\int_{B_{r}}(\snr{\bar{v}_{j}}^{2}-M^{2}_0)_{+}^d \, dx=0\,.
$$
Using the information in the last two displays, recalling the definition of $\mathcal{F}_{j}$ in \rif{ifunzionali}, and also using \rif{conv}$_2$, \rif{3.01}$_3$ and \rif{2}, we find 
$$
 \lim_{j\to \infty}\, \mathcal{F}_{j}(\bar{v}_{j},B_{r})= \lim_{j\to \infty}\, \mathcal{F}(\tilde u_j,B_{r})+\lim_{j\to \infty}\,\varepsilon_{j}\int_{B_{r}}(\snr{D\bar v_j}^{2}+\mu^{2})^{d} \, dx=\mathcal{F}(u,B_{r})\,.
 $$
Minimality of $u_j$, i.e., $ \mathcal{F}_{j}(u_{j},B_{r})\leq \mathcal{F}_{j}(\bar v_{j},B_{r})$, and the above display, then give
\eqn{6}
$$
\limsup_{j\to \infty}\mathcal{F}_{j}(u_{j},B_{r})\le\mathcal{F}(u,B_{r})\,,
$$
and therefore, up to relabelling, we can assume that $
\mathcal{F}_{j}(u_{j},B_{r}) \leq \mathcal{F}(u,B_{r})+1
$
holds for every $j\in \mathbb{N}$. 
This and \rif{assF}$_1$ imply
\begin{flalign}
\nu \int_{B_{r}}\snr{Du_j}^{p} \, dx+\varepsilon_{j}\int_{B_{r}}(\snr{Du_j}^{2}+\mu^{2})^{d} \, dx \le\mathcal{F}_{j}(u_{j},B_{r})\leq\mathcal{F}(u,B_{r})+1\label{20}\,.
\end{flalign} 
Poincar\'e inequality and \rif{conv}$_{3,4}$ yield
\begin{flalign}
\notag \|u_j\|_{L^p(B_r)} & \leq \|u_j-\tilde u_j\|_{L^p(B_r)}+ \|\tilde u_j\|_{L^p(B_r)} \\
&\leq 
cr\|Du_j\|_{L^p(B_r)} +cr\|D\tilde u_j\|_{L^p(B_r)} + cr^{n/p+\gamma}H \leq c[\mathcal{F}(u,B_{r})+1]^{1/p} +  cr^{n/p}H\,.\label{serve}
\end{flalign}
Therefore, up to a non relabelled sequence, we can assume that 
\begin{flalign}\label{7}
u_{j}\rightharpoonup v \quad \mbox{in} \ \ W^{1,p}(B_{r},\mathbb{R}^{N}), \ \ \mbox{for some $v$ such that  $v \in u+W^{1,p}_0(B_{r},\mathbb{R}^{N})$}\,.
\end{flalign}
Weak lower semicontinuity and \rif{6} then imply
\eqn{weak1}
$$
\mathcal{F}(v,B_{r}) \leq \liminf_{j\to \infty}\mathcal{F}(u_{j},B_{r})\leq \limsup_{j\to \infty}\mathcal{F}_{j}(u_{j},B_{r})\le\mathcal{F}(u,B_{r})\,,
$$
In turn, as $u-v\in W^{1,p}_{0}(B_{r},\mathbb{R}^{N})$, the minimality of $u$ renders that $\mathcal{F}(u,B_{r})\le \mathcal{F}(v,B_{r})$ and so $\mathcal{F}(u,B_{r})=\mathcal{F}(v,B_{r})$. The strict convexity of $z\mapsto F(\cdot,z)$ implied by $\eqref{assF}_{2}$ leads to 
\begin{flalign}\label{30}
u=v \ \ \mbox{almost everywhere on} \ \ B_{r}\,.
\end{flalign}
This means that \rif{weak1} now becomes
\begin{flalign*}
\mathcal{F}(u,B_{r})\le \liminf_{j\to \infty}\mathcal{F}(u_{j},B_{r})\le \limsup_{j\to \infty}\mathcal{F}_j(u_{j},B_{r})\leq \mathcal{F}(u,B_{r})\,,
\end{flalign*}
so that, recalling the definition of $\mathcal F_j$ in \rif{ifunzionali}, we gain
$$
\lim_{j\to \infty}\int_{\mathbb{R}^{n}}\int_{\mathbb{R}^{n}}\frac{\left(\snr{u_{j}(x)-u_{j}(y)}^{2}-M^{2}\snr{x-y}^{2\gamma}\right)_{+}^{d}}{\snr{x-y}^{n+2sd}} \, \dx\dy=
\lim_{j\to \infty}\int_{B_{r}}(\snr{u_{j}}^{2}-M^{2}_0)^{d}_{+} \dx =0\,.
$$
Up to a not relabelled subsequence, we can therefore assume that 
\begin{flalign}\label{40}
\begin{cases}
\displaystyle \int_{\mathbb{R}^{n}}\int_{\mathbb{R}^{n}}\frac{\left(\snr{u_{j}(x)-u_{j}(y)}^{2}-M^{2}\snr{x-y}^{2\gamma}\right)_{+}^{d}}{\snr{x-y}^{n+2sd}} \, \dx\dy\leq cr^{n+2d(\gamma-s)}H^{2d} \\ \\ 
\displaystyle
\int_{B_{r}}(\snr{u_{j}}^{2}-M^{2}_0)^{d}_{+} \dx\leq r^{n+2\gamma d} H^{2d}
\end{cases}
\end{flalign}
hold for all $j\in \N$. Recalling the definition of $M_0$ in \rif{1}$_2$, a direct consequence of \rif{40}$_2$ is
\eqn{sisisi}
$$
\|u_{j}\|_{L^{2d}(B_{r})}^{2d}\le cr^{n+2\gamma d} H^{2d}\,. 
$$
Then, since $u_j \equiv  \bar{v}_{j}$ outside $B_r$, as $u_j \in \mathbb X( \bar{v}_{j},B_r)$, and since in turn $v_j \equiv 0$ outside $B_{3r/2}$, by \rif{sisisi} and \rif{3.01}$_{1}$ we get 
\eqn{20.03}
$$
\|u_{j}\|_{L^{2d}(\er^n)}^{2d}\leq \|u_{j}\|_{L^{2d}(B_r)}^{2d}+\|v_{j}\|_{L^{2d}(B_{3r/2})}^{2d} \le cr^{n+2\gamma d} H^{2d}\,.
$$
Using \rif{40}$_1$ it now follows that
\begin{flalign}\label{bubu}
\notag [u_j]_{s, 2d;B_{2r}}^{2d}&\leq c\int_{B_{2r}}\int_{B_{2r}}\frac{(\snr{u_{j}(x)-u_{j}(y)}^{2}-M^{2}\snr{x-y}^{2\gamma})^{d}_{+}}{\snr{x-y}^{n+2sd}} \dx\dy\nonumber \\
&\qquad \ + cM^{2d}\int_{B_{2r}}\int_{B_{2r}}\frac{\dx\dy}{\snr{x-y}^{n-2d(\gamma-s)}}\nonumber \\
&\leq  cr^{n+2d(\gamma-s)}H^{2d}+  cM^{2d}r^{n+2d(\gamma-s)}=cr^{n+2d(\gamma-s)}H^{2d}
\end{flalign}
where $c\equiv c( \data_{\rm e})$. Using this last estimate, and splitting again as in \rif{20.02pre}, we have
\begin{flalign*}
[u_j]_{s, 2d;\er^n}^{2d}&=2\int_{\mathbb{R}^{n}\setminus B_{2r}}\int_{B_{2r}}\frac{\snr{u_{j}(x)}^{2d}}{\snr{x-y}^{n+2sd}} \dx\dy +[u_j]_{s, 2d;B_{2r}}^{2d}\nonumber \\
& =2\int_{\mathbb{R}^{n}\setminus B_{2r}}\int_{B_{3r/2}\setminus B_{r}}\frac{\snr{\bar{v}_{j}(x)}^{2d}}{\snr{x-y}^{n+2sd}} \dx\dy\\ & \qquad +2\int_{\mathbb{R}^{n}\setminus B_{2r}}\int_{B_{r}}\frac{\snr{u_{j}(x)}^{2d}}{\snr{x-y}^{n+2sd}} \dx\dy\nonumber+ [u_j]_{s, 2d;B_{2r}}^{2d}\\
&\le c\int_{\mathbb{R}^{n}\setminus B_{r/2}}\frac{\dz}{\snr{z}^{n+2sd}}\left(\|\bar{v}_{j}\|_{L^{2d}(B_{3r/2})}^{2d} +\|u_{j}\|_{L^{2d}(B_{r})}^{2d}\right)
+[u_j]_{s, 2d;B_{2r}}^{2d}\nonumber \,.
\end{flalign*}
Using $\eqref{3.01}_1$ and \eqref{20.03}-\rif{bubu} in the above display, we conclude with
\eqn{seminorm} 
$$ [u_{j}]_{s,2d;\mathbb{R}^{n}}^{2d} \le
cr^{n+2d(\gamma-s)}H^{2d}
$$
for $c\equiv c(\data_{\rm e})$. Recalling that $\beta_0=s-n/(2d)>0$ in \rif{beta0}, we get
\begin{flalign}
[u_{j}]_{0,\beta_0;\mathbb{R}^{n}}^{2d}= [u_{j}]_{0,\beta_0;B_{2r}}^{2d}\stackleq{prosca} c [u_{j}]_{s,2d;\mathbb{R}^{n}}^{2d}\stackrel{\eqref{seminorm}}{\le} cr^{n+2d(\gamma-s)}H^{2d}\,,\label{9}
\end{flalign}
for every $j \in \en$ and for a constant $c\equiv c(\data_{\rm e})$ which is independent of $j$. 
 From now on, we shall be using the following notation:
\eqn{laF}
$$
\mathds{F}(u,B_r):= \mathcal{F}(u,B_{r}) +r^{n}H^{2d}+1\,. 
$$
By \rif{20}-\rif{serve} we then have  (recall that $2d>p$ and therefore $H \leq H^{2d/p}+1$) 
\eqn{conse0}
$$
\|u_j\|_{W^{1,p}(B_r)} \leq c [\mathds{F}(u,B_r)]^{1/p}
$$
for $c\equiv c (\data)$, while 
\rif{seminorm}-\rif{9} and $r\leq 1$ imply
\eqn{conse}
$$
 [u_{j}]_{s,2d;\mathbb{R}^{n}}\leq c [\mathds{F}(u,B_r)]^{1/(2d)}\qquad \mbox{and} \qquad 
[u_{j}]_{s,2d;\mathbb{R}^{n}}^{2(d-1)}
[u_{j}]_{0,\beta_0;\mathbb{R}^{n}}^{2} \leq c\mathds{F}(u,B_r)\,,
$$
for $c\equiv c (\data_{\rm e})$. 
\subsection{Step 5: The Euler-Lagrange system}
We adopt the short notation
\eqn{adopt}
$$
F_{j}(x,z):=F(x,z)+\varepsilon_{j}(\snr{z}^{2}+\mu^{2})^{d}\,.
$$
The Euler-Lagrange system of the functional $\mathcal{F}_{j}$, computed for variations of the type $u_j +t\varphi$ for $\varphi\in W^{1,2d}_{0}(B_{r},\mathbb{R}^{N})$ and $t \in (-1, 1)$, is now
\begin{eqnarray}\label{el}
&& 0= \int_{B_{r}}\partial_{z}F_{j}(x,Du_{j})\cdot D\varphi \, dx+2d\int_{B_{r}}(\snr{u_{j}}^{2}-M^{2}_0)^{d-1}_{+}u_{j}\cdot\varphi \, \dx\nonumber \\
&&+2d\int_{\er^n}\int_{\er^n}\frac{\left(\snr{u_{j}(x)-u_{j}(y)}^{2}-M^{2}\snr{x-y}^{2\gamma}\right)^{d-1}_{+}(u_{j}(x)-u_{j}(y))\cdot (\varphi(x)-\varphi(y))}{\snr{x-y}^{n+2sd}} \, \dx\dy\,.
\end{eqnarray}
See also \cite{dkp} for the standard derivation concerning the nonlocal term. Equation \rif{el} is in fact  satisfied for all $\varphi\in W^{1,2d}_{0}(B_{r},\mathbb{R}^{N})$. For this, notice that $u_{j}\in W^{1,2d}(B_{r},\mathbb{R}^{N})$ and $F_{j}(\cdot)$ has standard $2d$-growth. Moreover, from Section \ref{st1},  \rif{luckybound} and \rif{range}, it is $2d>\tilde{p}>q$, and also recall that $
W^{1,2d}(B_{r},\mathbb{R}^{N})\hookrightarrow W^{t,2d}(B_{r},\mathbb{R}^{N}) $ for every $t\in (0,1)$, \cite[Proposition 2.2]{pala} so, recalling $\eqref{3.01}_{4}$, we can conclude that $\bar{v}_{j}+W^{1,2d}_{0}(B_{r},\mathbb{R}^{N})\hookrightarrow W^{t,2d}(\mathbb{R}^{n},\mathbb{R}^{N})$ again for all $t\in (0,1)$, and, in particular, for $t=s$. Now, we fix parameters 
\eqn{par1}
$$0<\rr\le \tau_{1}<\tau_{2}\le r$$
and set $\varphi:=\tau_{-h}(\eta^{2}\tau_{h}u_{j})$, which is admissible in \eqref{el}, as we take
\begin{flalign}\label{eta}
\begin{cases}
\eta\in C^{1}_{0}(B_{(3\tau_{2}+\tau_{1})/4})\\
\mathds{1}_{B_{(\tau_{2}+\tau_{1})/2}}\le \eta\le \mathds{1}_{B_{(3\tau_{2}+\tau_{1})/4}}\\
\snr{D\eta}\lesssim \frac{1}{\tau_{2}-\tau_{1}}
\end{cases}
\end{flalign}
and $h\in \mathbb{R}^{n}\setminus \{0\}$ is such that 
\begin{flalign}\label{boundh}
0< \snr{h} < \frac{\tau_{2}-\tau_{1}}{2^{10}}\,.
\end{flalign} Next, notice that by definitions in \rif{ttau1}-\rif{ttau2} we have 
\eqn{riscrivi}
$$
\left\{
\begin{array}{c}
\tau_{-h}(\eta^{2}\tau_{h}u_{j})(x)-\tau_{-h}(\eta^{2}\tau_{h}u_{j})(y) = \tilde \tau_{-h}\psi(x,y)
\\ [8 pt]
\psi(x,y):= \eta^{2}(x)(\tau_{h}u_{j})(x)-\eta^{2}(y)(\tau_{h}u_{j})(y)\,.
\end{array}
\right.
$$
Therefore, testing \eqref{el} with $\varphi$, and using integration-by-parts for finite difference operators, we obtain, rearranging the terms
\begin{flalign}
2d\int_{\er^n}\int_{\er^n}&\frac{\tilde \tau_{h}\left[\left(\snr{u_{j}(x)-u_{j}(y)}^{2}-M^{2}\snr{x-y}^{2\gamma}\right)^{d-1}_{+}(u_{j}(x)-u_{j}(y))\right]}{\snr{x-y}^{n+2sd}}\cdot\psi(x,y) \, \dx\dy\nonumber \\
&+2d\int_{B_{r}}\eta^{2}\tau_{h}\left((\snr{u_{j}}^{2}-M^{2})_{+}^{d-1}u_{j}\right)\cdot\tau_{h}u_{j} \, \dx\nonumber \\
&+\int_{B_{r}}\tau_{h}\partial_{z}F_{j}(x,Du_{j})\cdot\left[\eta^{2}\tau_{h}Du_{j}+2\eta \tau_{h}u_{j}\otimes D \eta\right] \, \dx=:\mbox{(I)}+\mbox{(II)}+\mbox{(III)}=0\,.
\label{completa}
\end{flalign}
\subsection{Step 6: Estimates for the nonlocal term $\mbox{(I)}$}\label{p2non}
For $\lambda\in [0,1]$, and with $h \in \er^n$ as in \rif{completa}, define
\begin{eqnarray*}
U_{\lambda,h}(x,y):=u_{j}(x)-u_{j}(y)+\lambda\ttau_{h}(u_{j}(x)-u_{j}(y))
\end{eqnarray*}
and set
\begin{eqnarray*}
&& \mathcal{A}(x,y,\lambda):=\left(\snr{U_{\lambda,h}(x,y)}^{2}-M^{2}\snr{x-y}^{2\gamma}\right)^{d-1}_{+}U_{\lambda,h}(x,y)\,.
\end{eqnarray*}
From this it follows that
\begin{eqnarray}\label{boh1}
\notag &&\hspace{-2mm}\frac{\d}{\d\lambda}\mathcal{A}(x,y,\lambda)\equiv \mathcal{A}'(x,y,\lambda)\\&& 
\ = 2(d-1)\left(\snr{U_{\lambda,h}(x,y)}^{2}-M_{2}^{2}\snr{x-y}^{2\gamma}\right)_{+}^{d-2}\left(U_{\lambda,h}(x,y)\cdot \tilde{\tau}_{h}(u_{j}(x)-u_{j}(y))\right)U_{\lambda,h}(x,y)\nonumber\\ 
&& \qquad \ \ +\left(\snr{U_{\lambda,h}(x,y)}^{2}-M_{2}^{2}\snr{x-y}^{2\gamma}\right)^{d-1}_{+}\ttau_{h}(u_{j}(x)-u_{j}(y))\,.
\end{eqnarray}
Notice that $U_{\lambda,h}(x,y) = - U_{\lambda,h}(y, x)$ and therefore 
\eqn{boh3}
$$
\mathcal{A}'(x,y,\lambda)=-\mathcal{A}'(y,x,\lambda)\,,$$
holds whenever $x, y \in \er^n$. Moreover, from \rif{riscrivi}$_2$, we have also
\eqn{boh2}
$$
\psi(x,y)=\eta^{2}(x)[ \tilde{\tau}_{h}(u_{j}(x)-u_{j}(y))] +[\eta^{2}(x)-\eta^{2}(y)]\tau_{h}u_{j}(y)\,.
$$
Identities in \rif{boh1} and \rif{boh2} allow to write
\begin{flalign}
\notag \mbox{(I)} &= 2d\int_{\er^n}\int_{\er^n}\left(\int_{0}^{1}\mathcal{A}'(x,y,\lambda)\d\lambda\right)\cdot\frac{\psi(x,y)}{\snr{x-y}^{n+2sd}} \dx \dy\nonumber \\
\notag&= 2d\int_{\er^n}\int_{\er^n}\eta^{2}(x)\left(\int_{0}^{1}\mathcal{A}'(x,y,\lambda)\d\lambda\right)\cdot\frac{\tilde{\tau}_{h}(u_{j}(x)-u_{j}(y))}{\snr{x-y}^{n+2sd}} \dx \dy\nonumber \\
\notag& \qquad  + 2d\int_{\er^n}\int_{\er^n}(\eta^{2}(x)-\eta^{2}(y))\left(\int_{0}^{1}\mathcal{A}'(x,y,\lambda)\d\lambda\right)\cdot\frac{\tau_{h}u_{j}(y)}{\snr{x-y}^{n+2sd}} \dx \dy\nonumber \\
&=:2d\int_{\er^n}\int_{\er^n}\int_{0}^{1}\frac{\eta^{2}(x)\mathcal{B}_{1}(x,y,\lambda)}{\snr{x-y}^{n+2sd}} \d\lambda \dx\dy\notag\\
& \qquad +2d\int_{\er^n}\int_{\er^n}\int_{0}^{1}(\eta^{2}(x)-\eta^{2}(y))\frac{\mathcal{B}_{2}(x,y,\lambda)}{\snr{x-y}^{n+2sd}} \d\lambda \dx\dy\,, \label{pezzo}
\end{flalign}
with obvious meaning of $\mathcal{B}_{1}(x,y,\lambda)$ and $\mathcal{B}_{2}(x,y,\lambda)$. 
We continue with the estimation of the two last integrals above. As for the former, we note that
\begin{flalign}\label{boh}
\mathcal{B}_{1}(x,y,\lambda)&=2(d-1)\left(\snr{U_{\lambda,h}(x,y)}^{2}-M^{2}\snr{x-y}^{2\gamma}\right)_{+}^{d-2}\snr{U_{\lambda,h}(x,y)\cdot \tilde{\tau}_{h}(u_{j}(x)-u_{j}(y))}^{2}\nonumber \\
&\qquad +\left(\snr{U_{\lambda,h}(x,y)}^{2}-M^{2}\snr{x-y}^{2\gamma}\right)^{d-1}_{+}\snr{\ttau_{h}(u_{j}(x)-u_{j}(y))}^{2}\ge 0
\end{flalign}
and therefore it is also
\eqn{defiI}
$$
\mathds I:=
2d\int_{\er^n}\int_{\er^n}\int_{0}^{1}\frac{\eta^{2}(x)\mathcal{B}_{1}(x,y,\lambda)}{\snr{x-y}^{n+2sd}} \d\lambda \dx\dy\geq 0\,.
$$
In order to estimate the last term appearing in \rif{pezzo}, we start splitting as follows (recall that $\eta\equiv 0$ outside $B_{(3\tau_{2}+\tau_{1})/4}$ by \rif{eta}):
\begin{flalign}
\notag &\int_{\er^n}\int_{\er^n}\int_{0}^{1}(\eta^{2}(x)-\eta^{2}(y))\frac{\mathcal{B}_{2}(x,y,\lambda)}{\snr{x-y}^{n+2sd}} \d\lambda \dx\dy \\
&\qquad = \int_{\mathbb{R}^{n}\setminus B_{\tau_{2}}} \int_{B_{\tau_{2}}} [\ldots] \dx \dy+ \int_{B_{\tau_{2}}}  \int_{\mathbb{R}^{n}\setminus B_{\tau_{2}}}[\ldots]\dx \dy + \int_{B_{\tau_{2}}}\int_{B_{\tau_{2}}} [\ldots] \dx \dy \notag \\
&\qquad =:\mbox{(I)}_{1}+\mbox{(I)}_{2}+\mbox{(I)}_{3}\label{spezza}\,. 
\end{flalign}
For the estimation of the last three pieces in \rif{spezza}, we rearrange \eqref{boh1} as follows:
\begin{flalign}
\notag \mathcal{A}'(x,y,\lambda)&=2(d-1)\left[\left(\snr{U_{\lambda,h}(x,y)}^{2}-M^{2}\snr{x-y}^{2\gamma}\right)_{+}^{(d-2)/2}\left(U_{\lambda,h}(x,y)\cdot \tilde{\tau}_{h}(u_{j}(x)-u_{j}(y))\right)\right]\nonumber \\
\notag &\qquad  \left(\snr{U_{\lambda,h}(x,y)}^{2}-M^{2}\snr{x-y}^{2\gamma}\right)_{+}^{(d-2)/2}U_{\lambda,h}(x,y)\nonumber \\
& \qquad \qquad  \qquad  +\left(\snr{U_{\lambda,h}(x,y)}^{2}-M^{2}\snr{x-y}^{2\gamma}\right)^{d-1}_{+}\ttau_{h}(u_{j}(x)-u_{j}(y))\;.\label{rearrange}
\end{flalign}
With such an expression at hand, and recalling that $\eta \equiv 0$ outside $B_{(3\tau_{2}+\tau_{1})/4}$ (see \rif{eta}), by H\"older and Young inequalities, we estimate
\begin{eqnarray}
\notag
\snr{\mbox{(I)}_{1}}&\leq  &
\int_{\mathbb{R}^{n}\setminus B_{\tau_{2}}} \int_{B_{\tau_{2}}} \frac{\eta^{2}(x)|\mathcal{B}_{2}(x,y,\lambda)|}{\snr{x-y}^{n+2sd}}  \dx \dy \\
\notag&\le& \frac1{6}\int_{\mathbb{R}^{n}\setminus B_{\tau_{2}}} \int_{B_{(3\tau_{2}+\tau_{1})/4}}\int_{0}^{1}\frac{\eta^{2}(x)\mathcal{B}_{1}(x,y,\lambda)}{\snr{x-y}^{n+2sd}} \d\lambda \dx\dy\\
\notag&& \quad +c\int_{\mathbb{R}^{n}\setminus B_{\tau_{2}}} \int_{B_{(3\tau_{2}+\tau_{1})/4}}\int_{0}^{1}\frac{\snr{U_{\lambda,h}(x,y)}^{2(d-1)}\snr{\tau_{h}u_{j}(y)}^{2}}{\snr{x-y}^{n+2sd}} \d\lambda\dx\dy\nonumber \\
\notag&\le&\frac{\mathds I}{6} +c[u_{j}]_{s,2d;\mathbb{R}^{n}}^{2(d-1)}\left(\int_{B_{(3\tau_{2}+\tau_{1})/4}}\int_{\mathbb{R}^{n}\setminus B_{\tau_{2}}}\frac{\snr{\tau_{h}u_{j}(y)}^{2d}}{\snr{x-y}^{n+2sd}} \dy\dx\right)^{1/d}\nonumber \\
\notag&\le&\frac{\mathds I}{6} +c[u_{j}]_{s,2d;\mathbb{R}^{n}}^{2(d-1)}
[u_{j}]_{0,\beta_0;\mathbb{R}^{n}}^{2}\left(\int_{B_{(3\tau_{2}+\tau_{1})/4}}\int_{\mathbb{R}^{n}\setminus B_{\tau_{2}}}\frac{\dx\dy}{\snr{x-y}^{n+2sd}} \right)^{1/d}\snr{h}^{2\beta_0}\nonumber \\
\notag&\le&\frac{\mathds I}{6} +c[u_{j}]_{s,2d;\mathbb{R}^{n}}^{2(d-1)}
[u_{j}]_{0,\beta_0;\mathbb{R}^{n}}^{2}\left(\int_{B_{(3\tau_{2}+\tau_{1})/4}}\int_{\mathbb{R}^{n}\setminus B_{(\tau_{2}-\tau_{1})/4}}\frac{\dz}{\snr{z}^{n+2sd}} \dx\right)^{1/d}\snr{h}^{2\beta_0}\nonumber \\
&\stackleq{conse}&\frac{\mathds I}{6}+\frac{c r^{n/d}\mathds{F}(u,B_r)}{(\tau_{2}-\tau_{1})^{2s}}\snr{h}^{2\beta_0}\,.
\label{similto}
\end{eqnarray}
Recall that $\mathds I$ has been defined in \rif{defiI}. For the estimate of $\mbox{(I)}_{2}$, we start observing that \rif{boh3} and \rif{boh} imply that 
\eqn{boh4}
$$
\mathcal{B}_{1}(x,y,\lambda)=\mathcal{B}_{1}(y,x,\lambda)$$
holds for every $x, y \in \er^n$.
Then, similarly to $\mbox{(I)}_{1}$, we find
\begin{eqnarray}
\notag \snr{\mbox{(I)}_{2}}&\leq  &
 \int_{B_{\tau_{2}}}\int_{\mathbb{R}^{n}\setminus B_{\tau_{2}}} \frac{\eta^{2}(y)|\mathcal{B}_{2}(x,y,\lambda)|}{\snr{x-y}^{n+2sd}}  \dx \dy \\
&\stackrel{\eqref{boh4}}{=} &
\int_{\mathbb{R}^{n}\setminus B_{\tau_{2}}} \int_{B_{\tau_{2}}} \frac{\eta^{2}(x)|\mathcal{B}_{2}(x,y,\lambda)|}{\snr{x-y}^{n+2sd}}  \dx \dy\stackleq{similto}\frac{\mathds I}{6}+\frac{c r^{n/d}\mathds{F}(u,B_r)}{(\tau_{2}-\tau_{1})^{2s}}\snr{h}^{2\beta_0}\,.
\label{similto1}
\end{eqnarray}
In order to estimate $\mbox{(I)}_{3}$ we preliminary write 
\begin{flalign}
\notag \mbox{(I)}_{3} & =  \int_{B_{\tau_{2}}}\int_{B_{\tau_{2}}}\int_{0}^{1}\eta(x) (\eta(x)-\eta(y))\frac{\mathcal{B}_{2}(x,y,\lambda)}{\snr{x-y}^{n+2sd}} \d\lambda \dx\dy 
\\
& \qquad + \int_{B_{\tau_{2}}}\int_{B_{\tau_{2}}}\int_{0}^{1}\eta(y) (\eta(x)-\eta(y))\frac{\mathcal{B}_{2}(x,y,\lambda)}{\snr{x-y}^{n+2sd}} \d\lambda \dx\dy  =: \mbox{(I)}_{4}+\mbox{(I)}_{5}\label{addi1}\,.
\end{flalign}
In turn, we estimate $\mbox{(I)}_{4}$; recalling \rif{rearrange} and \rif{eta}, using first Young and then H\"older inequality, we find
\begin{eqnarray}
\notag \snr{\mbox{(I)}_{4}}&\le & \int_{B_{\tau_{2}}}\int_{B_{\tau_{2}}}\int_{0}^{1}\eta(x) |\eta(x)-\eta(y)|\frac{|\mathcal{B}_{2}(x,y,\lambda)|}{\snr{x-y}^{n+2sd}} \d\lambda \dx\dy 
\\
\notag&\leq &
\frac{d}{6}\int_{B_{\tau_{2}}}\int_{B_{\tau_{2}}}\int_{0}^{1}\frac{\eta^{2}(x)\mathcal{B}_{1}(x,y,\lambda)}{\snr{x-y}^{n+2sd}} \d\lambda \dx\dy\\
\notag&& \quad + c\int_{B_{\tau_{2}}}\int_{B_{\tau_{2}}}\int_{0}^{1}\frac{\snr{U_{\lambda,h}(x,y)}^{2(d-1)}\snr{\eta(x)-\eta(y)}^2\snr{\tau_{h}u_{j}(y)}^{2}}{\snr{x-y}^{n+2sd}} \d\lambda\dx\dy\nonumber \\
\notag&\leq&\frac{\mathds I}{12}+c[u_{j}]_{s,2d;\er^n}^{2(d-1)}\left(\int_{B_{\tau_{2}}}\int_{B_{\tau_{2}}}\frac{\snr{\tau_{h}u_{j}(y)}^{2d}\snr{\eta(x)-\eta(y)}^{2d}}{\snr{x-y}^{n+2sd}} \dx\dy\right)^{1/d}\nonumber \\
\notag&\leq & \frac{\mathds I}{12}+ c[u_{j}]_{s,2d;\er^n}^{2(d-1)}[u_{j}]_{0,\beta_0;\mathbb{R}^{n}}^{2}\|D\eta\|_{L^\infty}^2\left(\int_{B_{\tau_{2}}}\int_{B_{\tau_{2}}}\frac{\dx\dy}{\snr{x-y}^{n+2(s-1)d}} \right)^{1/d}\nonumber \\
&\stackleq{conse} &\frac{\mathds I}{12}+\frac{c r^{n/d+2(1-s)}\mathds{F}(u,B_r)}{ (\tau_{2}-\tau_{1})^{2}}\snr{h}^{2\beta_0}
\label{similto2}
\end{eqnarray}
for $c\equiv c(\data_{\rm e})$. On the other hand, we also have 
\begin{eqnarray*}
\snr{\mbox{(I)}_{5}}&\le & \int_{B_{\tau_{2}}}\int_{B_{\tau_{2}}}\int_{0}^{1}\eta(y) |\eta(x)-\eta(y)|\frac{|\mathcal{B}_{2}(x,y,\lambda)|}{\snr{x-y}^{n+2sd}} \d\lambda \dx\dy 
\\
&\stackrel{\eqref{boh4}}{=} & \int_{B_{\tau_{2}}}\int_{B_{\tau_{2}}}\int_{0}^{1}\eta(x) |\eta(x)-\eta(y)|\frac{|\mathcal{B}_{2}(x,y,\lambda)|}{\snr{x-y}^{n+2sd}} \d\lambda \dx\dy \\
&\stackleq{similto2} &\frac{\mathds I}{12}+\frac{c r^{n/d+2(1-s)}\mathds{F}(u,B_r)}{ (\tau_{2}-\tau_{1})^{2}}\snr{h}^{2\beta_0}\,.
\end{eqnarray*}
Merging the content of the last three displays yields 
\eqn{stima33}
$$
\snr{\mbox{(I)}_{3}} \leq \frac{\mathds I}{6}+\frac{c r^{n/d+2(1-s)}\mathds{F}(u,B_r)}{ (\tau_{2}-\tau_{1})^{2}}\snr{h}^{2\beta_0}
$$
for $c\equiv c(\data)$, which is the final estimate for $\mbox{(I)}_{3}$. Collecting the estimates found in \rif{similto}, \rif{similto1} and \rif{stima33} to \rif{spezza}, yields
$$
\notag \left|2d\int_{\er^n}\int_{\er^n}\int_{0}^{1}(\eta^{2}(x)-\eta^{2}(y))\frac{\mathcal{B}_{2}(x,y,\lambda)}{\snr{x-y}^{n+2sd}} \d\lambda \dx\dy \right| \leq 
 \frac{\mathds I}{2} + \frac{c r^{n/d}\mathds{F}(u,B_r)}{(\tau_{2}-\tau_{1})^{2}}\snr{h}^{2\beta_0}\,.
$$
Finally, merging this last estimate with \rif{pezzo} and reabsorbing the term $\mathds I$, we conclude with 
\eqn{51}
$$
d\int_{\mathbb{R}^{n}}\int_{\mathbb{R}^{n}}\int_{0}^{1}\frac{\eta^{2}(x)\mathcal{B}_{1}(x,y,\lambda)}{\snr{x-y}^{n+2sd}} \d\lambda\dx\dy\leq  \mbox{(I)} + \frac{c r^{n/d}\mathds{F}(u,B_r)}{ (\tau_{2}-\tau_{1})^{2}}\snr{h}^{2\beta_0}\,,
$$
where $c\equiv c(\data_{\rm e})$. 
\subsection{Step 7: Estimates for the local terms $\mbox{(II)}$ and $\mbox{(III)}$}
Here we benefit from \cite{DM, ELM}. For $\mbox{(II)}$, exactly as in \cite[Section 3.2, (3.27)]{DM}, we have
\begin{eqnarray}
\mbox{(II)}&=&2d\int_{B_{R}}\eta^{2}\int_{0}^{1}\frac{d}{d\lambda}\left((\snr{u_{j}+\lambda\tau_{h}u_{j}}^{2}-M^{2}_0)_{+}^{d-1}(u_{j}+\lambda\tau_{h}u_{j})\right) \, \d \lambda \cdot  \tau_{h}u_{j}  \dx\nonumber 
\\
&=&4d(d-1)\int_{B_{R}}\eta^{2}\int_{0}^{1}(\snr{u_{j}+\lambda \tau_{h}u_{j}}^{2}-M^{2})^{d-2}_{+}((u_{j}+\lambda \tau_{h}u_{j})\cdot \tau_{h}u_{j})^{2} \, \d \lambda  \dx\nonumber\\
&&+2d\int_{B_{R}}\eta^{2}\int_{0}^{1}(\snr{u_{j}+\lambda \tau_{h}u_{j}}^{2}-M^{2}_0)^{d-1}_{+} \, \d\lambda \ \snr{\tau_{h}u_{j}}^{2}  \dx \ge 0\;.\label{52}
\end{eqnarray}
For (III), recalling the definition in \rif{adopt}, we have
\begin{flalign*}
\mbox{(III)}&=\int_{B_{r}}\tau_{h}\partial_{z}F(x,Du_{j})\cdot\left[2\eta  \tau_{h}u_{j}\otimes D\eta+\eta^{2}\tau_{h}Du_{j}\right] \, dx\nonumber \\
&\quad +2d\varepsilon_{j}\int_{B_{r}}\tau_{h}((\mu^{2}+\snr{Du_{j}}^{2})^{d-1}Du_{j})\cdot\left[2\eta  \tau_{h}u_{j}\otimes D\eta+\eta^{2}\tau_{h}Du_{j}\right] \, dx=:\mbox{(III)}_{1}+\mbox{(III)}_{2}\,.
\end{flalign*}
We can then proceed as in \cite[Section 3.2]{DM}. 
By $\eqref{assF}_{2,3}$, a monotonicity estimate as the one in \cite[Section 3.2, term (I)$_j$]{DM} gives (recall that no $x$-dependence appears in this term)
\begin{flalign}
\notag \mbox{(III)}_{1} &\ge \frac 1c\int_{B_{r}}\eta^{2}(\snr{Du_{j}(x+h)}^{2}+\snr{Du_{j}(x)}^{2}+\mu^{2})^{(p-2)/2}\snr{\tau_{h}Du_{j}}^{2}  \, dx\\
&\qquad -\frac{c\snr{h} }{\tau_{2}-\tau_{1}}\int_{B_{\tau_{2}}}(\snr{Du_{j}}^{2}+1)^{q/2} \, dx\,,\label{rep1}
\end{flalign}
for $c\equiv c(n,N,\nu,L,p,q)$. Similarly, we have 
\begin{flalign}
\notag \mbox{(III)}_{2} &\ge \frac{\varepsilon_{j}}{c}\int_{B_{r}}\eta^{2}(\snr{Du_{j}(x+h)}^{2}+\snr{Du_{j}(x)}^{2}+\mu^{2})^{d-1}\snr{\tau_{h}Du_{j}}^{2}  \, dx\\
&\qquad -\frac{c\eps_j\snr{h}}{\tau_{2}-\tau_{1}}\int_{B_{\tau_{2}}}(\snr{Du_{j}}^{2}+1)^{d} \, dx\,.\label{rep1bis}
\end{flalign}
Connecting the inequalities in the last three displays we conclude with
\begin{flalign*}
\notag &\int_{B_{r}}\eta^{2}(\snr{Du_{j}(x+h)}^{2}+\snr{Du_{j}(x)}^{2}+\mu^{2})^{(p-2)/2}\snr{\tau_{h}Du_{j}}^{2}  \, dx \\&\qquad \leq  c\mbox{(III)}+\frac{c\snr{h}^{\alpha}}{\tau_{2}-\tau_{1}}\int_{B_{\tau_{2}}}(\snr{Du_{j}}^{2}+1)^{q/2} \, dx+ \frac{c\eps_j\snr{h}}{\tau_{2}-\tau_{1}}\int_{B_{\tau_{2}}}(\snr{Du_{j}}^{2}+1)^{d} \, dx\,,
\end{flalign*}
for $c\equiv c(n,N,\nu,L,p,q)$. Using \rif{20} in the last inequality and recalling \rif{laF}, finally leads to
\begin{flalign}
\notag &\int_{B_{r}}\eta^{2}(\snr{Du_{j}(x+h)}^{2}+\snr{Du_{j}(x)}^{2}+\mu^{2})^{(p-2)/2}\snr{\tau_{h}Du_{j}}^{2}  \, dx \\&\qquad \leq  c\mbox{(III)}+\frac{c\snr{h}^{\alpha}}{\tau_{2}-\tau_{1}}\left(
\mathds{F}(u,B_r)+\nr{Du_{j}}_{L^{q}(B_{\tau_{2}})}^q+1\right)\label{12}
\end{flalign}
for $c\equiv c(n,N,\nu,L,p,q) $. 
\subsection{Step 8: Local/nonlocal Caccioppoli type inequality}
Connecting \eqref{51}, \eqref{52} and \eqref{12} to \rif{completa}, and recalling $\eqref{eta}_{2}$, we end up with
\begin{flalign}\label{53}
\notag &\int_{B_{(\tau_{2}+\tau_{1})/2}}\int_{B_{(\tau_{2}+\tau_{1})/2}}\int_{0}^{1}\frac{\mathcal{B}_{1}(x,y,\lambda)}{\snr{x-y}^{n+2sd}} \d\lambda\dx\dy \\
& \qquad \quad +\int_{B_{(\tau_{2}+\tau_{1})/2}}(\snr{Du_{j}(x+h)}^{2}+\snr{Du_{j}(x)}^{2}+\mu^{2})^{(p-2)/2}\snr{\tau_{h}Du_{j}}^{2}  \dx \notag\\
& \qquad \qquad \quad \le \frac{c\snr{h}^{\alpha_0}}{(\tau_{2}-\tau_{1})^{2}}\left(\mathds{F}(u,B_r)+\nr{Du_{j}}_{L^{q}(B_{\tau_{2}})}^q+1\right)\,,
\end{flalign}
where $c\equiv c(\data)$ and $\alpha_0$ has been defined in \rif{beta1}.
\subsection{Step 9: Iteration and a priori estimate for $p\ge 2$}\label{hsec1}
Recalling the non-negativity in \rif{defiI}, estimate \eqref{53} with $p\ge 2$ yields
\eqn{analoga}
$$
\int_{B_{(\tau_{2}+\tau_{1})/2}}\snr{\tau_h Du_{j}}^{p} \, \dx \le \frac{c\snr{h}^{\alpha_0}}{(\tau_{2}-\tau_{1})^{2}}\left(\mathds{F}(u,B_r)+\nr{Du_{j}}_{L^{q}(B_{\tau_{2}})}^q+1\right)\,.
$$
Lemma \ref{l2} and \rif{boundh} now provide that 
\eqn{analoga2}
$$
u_{j}\in W^{1+\beta/p,p}(B_{(\tau_{2}+\tau_{1})/2},\mathbb{R}^{N})\cap W^{s,2d}(B_{r},\mathbb{R}^{N})
$$
holds together with (via \rif{viavia}, using \rif{analoga})
\begin{flalign}
\nr{u_{j}}_{W^{1+\beta/p,p}(B_{(\tau_{2}+\tau_{1})/2})}&\le \frac{c}{(\tau_{2}-\tau_{1})^{(n+\beta)/p}}\nr{u_{j}}_{W^{1,p}(B_{\tau_{2}})}\nonumber \\
&\qquad +\frac{c}{(\tau_{2}-\tau_{1})^{(2+\beta-\alpha_0)/p}}\left([\mathds{F}(u,B_r)]^{1/p}+\nr{Du_{j}}_{L^{q}(B_{\tau_{2}})}^{q/p}+1\right)\nonumber \\
&\leq \frac{c}{(\tau_{2}-\tau_{1})^{(n+\beta)/p}}\left([\mathds{F}(u,B_r)]^{1/p}+\nr{Du_{j}}_{L^{q}(B_{\tau_{2}})}^{q/p}+1\right)\,.\label{14}
\end{flalign}
Here $\beta <\alpha_0$ has been fixed in \rif{beta1}, and $c\equiv c(\data)\geq 1$, recall \rif{idati}; in the last stimate we have also used \rif{conse0}. Now we apply Lemma \ref{fraclem} with parameters $s_{1}=s$, $s_{2}=1+\beta/p$, $a= 2d$ and $t=p$ to get
\eqn{analoga3}
$$
\nr{Du_{j}}_{L^{\tilde{p}}(B_{\tau_{1}})}\le \frac{c}{(\tau_{2}-\tau_{1})^{\kappa}}[u_{j}]_{s,2d;B_{\tau_{2}}}^{\theta}\nr{Du_{j}}_{W^{\beta/p,p}(B_{(\tau_{2}+\tau_{1})/2})}^{1-\theta}
$$
for 
\eqn{iltheta}
$$\theta=\frac\beta{p(1-s)+\beta}\,,$$ 
$ \tilde p$ as in \rif{tp}, 
and where $c,\kappa\equiv c,\kappa(\data_{\rm e})$ are derived from Lemma \ref{fraclem}. By using 
\eqref{conse} and \eqref{14} in \rif{analoga3}, we obtain
\begin{flalign}\label{15}
\notag \nr{Du_{j}}_{L^{\tilde{p}}(B_{\tau_{1}})} & \le
   \frac{c}{(\tau_{2}-\tau_{1})^{(n+\beta)/p+\kappa}}
[\mathds{F}(u,B_r)]^{\frac{1-\theta}{p}+\frac{\theta}{2d}}
\\ &\qquad + 
 \frac{c}{(\tau_{2}-\tau_{1})^{(n+\beta)/p+\kappa}}
[\mathds{F}(u,B_r)]^{\frac{\theta}{2d}} \nr{Du_{j}}_{L^{q}(B_{\tau_{2}})}^{\frac{q(1-\theta)}{p}}\,.
\end{flalign}
By \rif{range} we can use the interpolation inequality
\begin{flalign}\label{18}
\nr{Du_{j}}_{L^{q}(B_{\tau_{2}})}\le \nr{Du_{j}}_{L^{\tilde{p}}(B_{\tau_{2}})}^{\theta_{1}}\nr{Du_{j}}_{L^{p}(B_{\tau_{2}})}^{1-\theta_{1}},
\end{flalign}
where $\theta_{1}\in (0,1)$ satisfies the identity
\begin{flalign}\label{sig}
\frac{1}{q}=\frac{\theta_{1}}{\tilde p}+\frac{1-\theta_{1}}{p} \ \Longrightarrow \ \theta_{1}=\frac{\tilde{p}(q-p)}{q(\tilde{p}-p)}\,.
\end{flalign}
Plugging \eqref{18} in \eqref{15} yields
\begin{flalign}
\notag \nr{Du_{j}}_{L^{\tilde{p}}(B_{\tau_{1}})}&\le  \frac{c}{(\tau_{2}-\tau_{1})^{(n+\beta)/p+\kappa}}
[\mathds{F}(u,B_r)]^{\frac{1-\theta}{p}+\frac{\theta}{2d}}\\
& \qquad  + 
\frac{c}{(\tau_{2}-\tau_{1})^{(n+\beta)/p+\kappa}}
[\mathds{F}(u,B_r)]^{\frac{\theta}{2d}} \nr{Du_{j}}_{L^{\tilde{p}}(B_{\tau_{2}})}^{\frac{\theta_{1}q(1-\theta)}{p}}\nr{Du_{j}}_{L^{p}(B_{\tau_{2}})}^{\frac{(1-\theta_{1})q(1-\theta)}{p}}\,,
\label{19}
\end{flalign}
again with $c\equiv c(\data)$, $\kappa\equiv \kappa(\data_{\rm e})$.  
Notice that
\begin{flalign}
\frac{q\theta_1(1-\theta)}{p}\stackrel{\rif{iltheta}, \rif{sig}}{=}\frac{\tilde{p}(s,d,\beta)(q-p)}{\tilde{p}(s,d,\beta)-p}\frac{1-s}{p(1-s)+\beta}\stackrel{\rif{270}}{<}1\,.\label{27}
\end{flalign}
This allows to apply Young inequality in \rif{19} with conjugate exponents
\begin{flalign*}
\left(\frac{p}{q\theta_{1}(1-\theta)},\frac{p}{p-q\theta_{1}(1-\theta)}\right)
\end{flalign*}
in order to get, using also $\eqref{assF}_{1}$ and \eqref{20},
\begin{flalign*}
\nr{Du_{j}}_{L^{\tilde{p}}(B_{\tau_{1}})}\le \frac{1}{2}\nr{Du_{j}}_{L^{\tilde{p}}(B_{\tau_{2}})}+\frac{c}{(\tau_{2}-\tau_{1})^{\kappa_1}}\left[\mathds{F}(u,B_r)\right]^{\tilde \kappa}\,,
\end{flalign*}
for suitable exponents $\kappa_1,  \kappa_2$ depending on $\data_{\rm e}$ and $c\equiv c(\data)$. Lemma \ref{itercz} together with the content of the previous display finally gives
$$
\nr{Du_{j}}_{L^{\tilde{p}}(B_{\rr})}\le\frac{c}{(r-\rr)^{\kappa_1}}\left[\mathds{F}(u,B_r)\right]^{\tilde \kappa}\,,
$$
for $c\equiv c(\data)$ and $\kappa_{1},\kappa_{2}\equiv \kappa_1,\kappa_{2}(\data_{\rm e})$. Recalling \rif{laF} and that $\tilde p > \mathfrak{q}$, this implies
$$
\nr{Du_j}_{L^{\mathfrak{q}}(B_{\rr})}\le \frac{c}{(r-\rr)^{\kappa_{1}}}\left(\mathcal{F}(u,B_{r}) +r^{n}H^{2d}+1\right)^{\tilde \kappa}\,.
$$
By using \rif{7} and \rif{30}, the assertion in \rif{ine} now follows by lower semicontinuity with $\kappa_2 = 2\tilde \kappa d $. This concludes the proof of Theorem \ref{t1} in the case $p\geq 2$. 
\begin{remark} {\em From the proof above and the choice of the parameters in Step 1, it is not difficult to see that when 
$$q,\mathfrak q \to p+ \frac{\min\{\alpha,2\gamma\}}{ \vartheta(1-\gamma)}\,,$$
thereby forcing $\tilde p$ to approach the same value, then $d\to \infty$. This makes the exponent $\kappa_2$ in the final estimate \rif{ine} blow-up. This is a typical phenomenon in regularity for $(p,q)$-growth functionals, already encountered in the form of the a priori estimates found in \cite{DM, ELM, M1}.}
\end{remark}
\subsection{Step 10: Iteration and a priori estimate for $1<p<2$}\label{hsec2}
 In this case we use H\"older inequality to estimate
\begin{flalign}\label{22}
\int_{B_{(\tau_{2}+\tau_{1})/2}}\snr{\tau_{h}Du_{j}}^{p} \dx &\le\left(\int_{B_{(\tau_{2}+\tau_{1})/2}}(\snr{Du_{j}(x+h)}^{2}+\snr{Du_{j}(x)}^{2}+\mu^{2})^{\frac{p-2}{2}} \snr{\tau_{h}Du_{j}}^{2}\dx\right)^{p/2}\nonumber \\
&\qquad \quad \cdot\left(\int_{B_{(\tau_{2}+\tau_{1})/2}}(\snr{Du_{j}(x+h)}^{2}+\snr{Du_{j}(x)}^{2}+\mu^{2})^{p/2} \dx\right)^{1-p/2}\nonumber \\
&\stackleq{conse0} c\left(\int_{B_{(\tau_{2}+\tau_{1})/2}}(\snr{Du_{j}(x+h)}^{2}+\snr{Du_{j}(x)}^{2}+\mu^{2})^{p/2}  \dx\right)^{p/2} \notag\\
& \qquad \quad \cdot \left[\mathds{F}(u,B_r)\right]^{1-p/2}\nonumber \\
&\stackleq{53} \frac{c\snr{h}^{\alpha_0p/2}}{(\tau_{2}-\tau_{1})^{p}}\left(\mathds{F}(u,B_r)+\nr{Du_{j}}_{L^{q}(B_{\tau_{2}})}^q+1\right)\,,
\end{flalign}
where $c \equiv c(\data)$. This is the analog of \rif{analoga} in the case $p<2$. We can therefore proceed as in the case $p\geq 2$; we report some of the details in the following as different bounds are coming up. As for \rif{analoga2}, Lemma \ref{l2} gives $
u_{j}\in W^{1+\beta/2,p}(B_{(\tau_{2}+\tau_{1})/2},\mathbb{R}^{N})\cap W^{s,2d}(B_{r},\mathbb{R}^{N})$, for $\beta$ as in \rif{beta1}, 
with the estimate (thanks to \rif{22})
\begin{flalign}
\nr{u_{j}}_{W^{1+\beta/2,p}(B_{(\tau_{2}+\tau_{1})/2})}&\le\frac{c}{(\tau_{2}-\tau_{1})^{n/p+\beta/2}}\nr{u_{j}}_{W^{1,p}(B_{\tau_{2}})}\nonumber \\
& \qquad +\frac{c}{(\tau_{2}-\tau_{1})^{1+(\beta-\alpha_0)/2}}
\left([\mathds{F}(u,B_r)]^{1/p}+\nr{Du_{j}}_{L^{q}(B_{\tau_{2}})}^{q/p}+1\right)\nonumber 
\\
&\leq\frac{c}{(\tau_{2}-\tau_{1})^{n/p+\beta/2}}\left([\mathds{F}(u,B_r)]^{1/p}+\nr{Du_{j}}_{L^{q}(B_{\tau_{2}})}^{q/p}+1\right)\,,
\label{23}
\end{flalign}
with $c\equiv c(\data)$. It is now the turn of Lemma \ref{fraclem} applied with parameters $s_{1}=s$, $s_{2}=1+\beta/2$, $a= 2d$ and $t=p$ to get
\eqn{24}
$$
\nr{Du_{j}}_{L^{\tilde{p}}(B_{\tau_{1}})}\le\frac{c}{(\tau_{2}-\tau_{1})^{\kappa}}[u_{j}]^{\theta}_{s,2d;B_{\tau_{2}}}\nr{Du_{j}}_{W^{\beta/2,p}(B_{(\tau_{2}+\tau_{1})/2})}^{1-\theta}
$$
where this time it is
\eqn{iltheta2}
$$\theta:=\frac{\beta}{2(1-s)+\beta}\,,$$ and where $ \tilde p$ is as in \rif{tp}, 
and $c,\kappa\equiv c,\kappa(\data_{\rm e})$ are derived from Lemma \ref{fraclem}. Plugging  \eqref{conse} and \eqref{23}  into \eqref{24}, we again arrive at the following analog of \rif{15}:
\begin{flalign}\label{15bis}
\notag \nr{Du_{j}}_{L^{\tilde{p}}(B_{\tau_{1}})} & \le
   \frac{c}{(\tau_{2}-\tau_{1})^{n/p+\beta/2+\kappa}}
[\mathds{F}(u,B_r)]^{\frac{1-\theta}{p}+\frac{\theta}{2d}}
\\ &\qquad + 
 \frac{c}{(\tau_{2}-\tau_{1})^{n/p+\beta/2+\kappa}}
[\mathds{F}(u,B_r)]^{\frac{\theta}{2d}} \nr{Du_{j}}_{L^{q}(B_{\tau_{2}})}^{\frac{q(1-\theta)}{p}}\,.
\end{flalign}
 Next, 
keeping in mind that $\tilde{p}>q$ by \rif{range}, we can apply the interpolation inequality \rif{18} 
 in \rif{15bis}, with $\theta_{1}\in (0,1)$ as in \eqref{sig} (but with the different, current definition of $\tilde p$). As for the case $p\geq 2$, using \rif{18} in \rif{15} yields \rif{19}, with
$$
\frac{q\theta_1(1-\theta)}{p}\stackrel{ \rif{sig},\rif{iltheta2}}{=}\frac{\tilde{p}(s,d,\beta)(q-p)}{\tilde{p}(s,d,\beta)-p}\frac{2(1-s)}{p[2(1-s)+\beta]}\stackrel{\rif{280}}{<}1\,.
$$
We can therefore proceed exactly as after \rif{27} to arrive at \rif{ine} and the proof of Theorem \ref{t1} is finally complete. 

\subsection{Step 11: Proof of Corollaries \ref{c0},\ref{c1}}\label{convsec0} In both cases the proof relies on the application of Theorem \ref{t1}, where we verify that $\mathcal {L}_{F,H}(u,B_{r})=0$ holds for every ball $B_r \Subset \Omega$, for suitable choices of $H >0$. We fix $B_r \Subset \Omega$ such that $ [u]_{0, \gamma;B_r}>0$ (otherwise the assertion is trivial) and we start from Corollary \ref{c1}, where we use \cite[Theorem 4]{BCM} in order to get a sequence $\{\tilde{u}_{j}\}$ of $W^{1,\infty}(B_r)$-regular maps such that 
\eqn{convy}
$$\int_{B_r} \left[|D\tilde{u}_{j}|^p+a(x)|D\tilde{u}_{j}|^{q}\right]\dx \to \int_{B_r} \left[|Du|^p+a(x)|Dw|^{q}\right]\dx\,,$$
and $\tilde u_j \to u$ strongly in $W^{1,p}(B_r)$. The sequence $\{\tilde{u}_{j}\}$ has been constructed in \cite{BCM} via mollification, i.e., $\tilde u_{j} := u*\rho_{\eps_j}$. Here, with 
$0<\delta \leq \min\{{\rm  dist}(B_r, \partial \Omega)/20 \}$ being fixed, $\{\eps_j\}$ can be taken as a sequence such that $\eps_j \to 0$, $0< \eps_j \leq\delta$. Moreover, $\{\rho_{\eps}\}$ is a family of standard mollifiers generated by a smooth,  non-negative and radial function 
$\rho \in C^\infty_0(B_1)$ such that $\|\rho\|_{L^1(\er^n)}=1$, via 
$\rho_{\eps}(x) = \eps^{-n} \rho(x/\eps)$. It follows that 
\eqn{sceltaH}
$$[\tilde{u}_{j}]_{0, \gamma;B_r} \leq [u]_{0, \gamma;(1+\delta)B_r}\,,$$
for every $j \in \en$. Proposition \ref{approssi1} then implies $\mathcal {L}_{F,H}(u,B_{r})=0$ for $H=[u]_{0, \gamma;(1+\delta)B_r}$. Therefore Theorem \ref{t1} applies, with \rif{ine} that translates into
\eqn{deltaine}
$$
\nr{Du}_{L^{\mathfrak{q}}(B_{\rr})}\le \frac{c}{(r-\rr)^{\kappa_{1}}}\left([\mathcal{F}(u,B_{r})]^{1/p}+[u]_{0, \gamma;(1+\delta)B_r}+1\right)^{\kappa_{2}}\;,
$$
which holds whenever $B\varrho\Subset B_r$ is concentric to $B_r$.   
Finally, letting $\delta \to 0$ in the above inequality leads to \rif{stima2} and the proof of Corollary \ref{c1} is complete. We only notice that we can invoke \cite[Theorem 4]{BCM} as we are assuming \rif{pq}, which is a more restrictive condition than the one considered in \rif{asy0}, which is in turn sufficient to prove \rif{convy}, as shown in \cite{BCM}. For Corollary \ref{c0}, the proof is totally similar and uses the same convolution argument of Corollary \ref{c1} to find the approximating sequence $\{\tilde u_j\}\subset W^{1,\infty}(B_r)$ such that 
the approximation in energy $\mathcal F(\tilde u_j, B_r) \to 
\mathcal F(u, B_r)$ and \rif{sceltaH} hold. In this case the proof of the approximation in energy is directly based on Jensen's inequality and the convexity of $G(\cdot)$. The details can be found in \cite[Lemma 12]{ELM}. No bound on the gap $q/p$ is required at this stage and \rif{pq} is only needed to refer to Theorem \ref{t1}, i.e., to prove the a priori estimates. The rest of the proof then proceeds as the one for Corollary \ref{c1}. 

\subsection{Step 12: Proof of Theorem \ref{t2}}\label{convsec}
This follows with minor modifications from the proof of Theorem \ref{t1}, that we briefly outline here. 
As the integrand $F(\cdot)$ is now $x$-independent and convex, we are in the situation of Corollary \ref{c0}, where we can verify \rif{orli1} with $b(\cdot)\equiv 1$ and $G(\cdot)\equiv F(\cdot)$. It follows that $\mathcal {L}_{F,H}(u,B_{r})=0$ with the choice $H=[u]_{0, \gamma;(1+\delta)B_r}$, where $\delta>0$ can be chosen arbitrarily small as in the proofs of Corollaries \ref{c0},\ref{c1}. 
As for the part concerning the a priori estimates, the proof follows the one given for Theorem \ref{t1} verbatim, once replacing, everywhere, $\alpha$ by $1$. With the current choice of $H$, this leads to establish \rif{deltaine}, from which 
\rif{stima2} finally follows letting $\delta \to 0$. 
\section{More on the autonomous case}\label{finalsec}
In the autonomous case $F(x, Dw)\equiv F(Dw)$, considering assumptions on the Hessian of  $F(\cdot)$ of the type in \rif{seconde}, leads to bounds that are better than \rif{pqa}. Specifically, we assume that $F \colon \er^{N\times n} \to \er$ is locally $C^2$-regular in $\er^{N\times n}\setminus\{0\}$ and satisfies
\eqn{asp1}
$$
\left\{
\begin{array}{c}
\nu (|z|^2+\mu^2)^{p/2} \leq F(z) \leq L(|z|^2+\mu^2)^{q/2}+L(|z|^2+\mu^2)^{p/2}\\[5 pt]
(|z|^2+\mu^2) |\partial_{zz} F(z)| \leq L(|z|^2+\mu^2)^{q/2}+L(|z|^2+\mu^2)^{p/2}\\ [5 pt]
\nu (|z|^2+\mu^2)^{(p-2)/2}|\xi|^2\leq    \partial^ 2F(z)\xi\cdot \xi\;, 
 \end{array}\right.
$$
for every choice of $z, \xi \in \er^{N\times n}$ such that $|z|\not=0$, and for exponents $1 \leq p \leq q$. As usual, $0<\nu\leq 1 \leq L$ are fixed ellipticity constants and $\mu \in [0,1]$. Such assumptions are for instance considered in \cite{BM, M1}. Using \rif{asp1}, instead of \rif{assF}, in the statement of Theorem \ref{t2} we can replace the right-hand side quantity in \rif{pqa} and \rif{tipau} by the larger
\eqn{leads to}
$$
p+\frac{2\gamma}{\vartheta(1-\gamma)}\,.
$$
The proof can be obtained along the lines of the one for Theorem \ref{t1}, using different, and actually more standard estimates in Step 6. In particular, estimates \rif{rep1}-\rif{rep1bis} can be replaced by 
\begin{flalign*}
\mbox{(III)}_{1}+\mbox{(III)}_{2}&\ge \frac 1c\int_{B_{r}}\eta^{2}(\snr{Du_{j}(x+h)}^{2}+\snr{Du_{j}(x)}^{2}+\mu^{2})^{(p-2)/2}\snr{\tau_{h}Du_{j}}^{2}  \, dx\\
& \qquad -\frac{c\snr{h}^{2}}{(\tau_{2}-\tau_{1})^2}\int_{B_{\tau_{2}}}(\snr{Du_{j}}^{2}+1)^{q/2} \, dx-\frac{c\eps_j \snr{h}^{2}}{(\tau_{2}-\tau_{1})^2} \int_{B_{\tau_{2}}}(\snr{Du_{j}}^{2}+1)^{d} \, dx\,,
\end{flalign*}
and therefore estimate \rif{12} can be replaced by
\begin{flalign*}
&\int_{B_{r}}\eta^{2}(\snr{Du_{j}(x+h)}^{2}+\snr{Du_{j}(x)}^{2}+\mu^{2})^{(p-2)/2}\snr{\tau_{h}Du_{j}}^{2}  \, \dx \\
& \qquad\qquad   \leq c\mbox{(III)}+\frac{c\snr{h}^{2}}{(\tau_{2}-\tau_{1})^2}\left(
\mathcal{F}(u,B_{r})+\nr{Du_{j}}_{L^{q}(B_{\tau_{2}})}^q+1\right)\,.
\end{flalign*}
This leads to change the parameters in Step 1. Specifically, we replace everywhere the quantity $ \min\{\alpha,2\gamma\}$, by $ 2\gamma=\min\{2,2\gamma\}$ starting from \rif{scelta iniziale}, while in \rif{beta1} we define $
\alpha_0:=2 \beta_0< 2 \gamma$. The rest of the proof follows unaltered and leads to establish \rif{ine} using now the bound in \rif{leads to}. 

\end{document}